\documentclass[12pt]{article}
\usepackage[cp1251]{inputenc}
\usepackage[T2A]{fontenc}
\usepackage{indentfirst}
\usepackage{amsmath}
\usepackage{graphicx}
\usepackage{float}
\usepackage{amsfonts}
\usepackage{amssymb}
\usepackage{amscd}
\usepackage{mathtext,amsmath, amstext, amsgen, amsthm} 
\usepackage{amsopn, amsfonts, mathrsfs, euscript,amscd,amssymb,latexsym}
\theoremstyle{definition}

\newtheorem{lemma}{Lemma}[section]
\newtheorem{theorem}{Theorem}[section]
\newtheorem{collor}{Corollary}[section]
\newtheorem{proposition}{Proposition}[section]
\newtheorem{remark}{Remark}[section]

\numberwithin{equation}{section}
\begin{document}

\title{Non-alternating Hamiltonian Lie algebras\\
 in characteristic 2. I}
\author{A.V. Kondrateva, M.I. Kuznetsov, N.G. Chebochko}
\date{}
\maketitle
\begin{center}
{\it {National Research Lobachevsky State University of Nizhny Novgorod, Russia}}\\
e-mail:
alisakondr@mail.ru, kuznets-1349@yandex.ru\\
{\it {National Research University Higher School of Economics,\\ Nizhny Novgorod, Russia}}\\
e-mail:
chebochko@mail.ru
\end{center}

\noindent ABSTRACT. The classification of graded non-alternating Hamiltonian Lie algebras over perfect field of characteristic 2 is obtained. It is shown that the filtered deformations of such algebras correspond to non-alternating Hamiltonian forms with polynomial coefficients in divided powers. It is proved that graded non-alternating Hamiltonian algebras are rigid with respect to filtered deformations except for some cases when the number of variables is 2, 3, 4 or when the heights of some variables are equal to 1.\newline
\newline
\newline
\noindent ACKNOWLEDGEMENTS. The authors thank S.M. Skryabin for sending them his papers on modular Lie algebras of Cartan type. The investigation is funded by RFBR according to research project N 18-01-00900\newline
\newline
\newline

I. Kaplansky in the work ~ \cite{Kap} gave examples of exceptional simple Lie algebras of characteristic 2. The third series of constructed algebras, denoted as Kap 3, is similar to the construction of Block algebras ~ \cite{Bl}, in which the skew-symmetric form is replaced by a non-alternating form . As it is well known, Blok algebras over a field of characteristic $ p $ are isomorphic to Hamiltonian Lie algebras over an algebra of divided powers ~ \cite{W1}, ~ \cite{KKu}. L. Lin ~ \cite{N} constructed a series of Hamiltonian Lie algebras of characteristic 2 with the simplest symmetric Poisson bracket $ \{f, ~ g \} = \sum \partial_if \partial_ig $ on the algebra of divided degrees and showed that the Lie algebras Kap 3 correspond to the case when the heights of the variables are equal to 1. Over the field of zero characteristic, this bracket occurs in finite-dimensional Hamiltonian Lie superalgebras $ H (n). $ Graded non-alternating Hamiltonian Lie algebras were intensively investigated by D. Leites, S. Bouarrouj, U. Yier, M. Messaoudene , P. Grozman, A. Lebedev, I. Schepochkina in the direction of disseminating ideas and methods
Lie superalgebra theory for the case of Lie algebras of even characteristic. A complex of symmetric differential forms in divided powers was constructed, which led to a more natural definition of non-alternating Hamiltonian Lie algebras within the framework of Hamiltonian formalism, an analysis of graded algebras from the point of view of Cartan prolongations was carried out, some of Volichenko algebras were considered (see ~ \cite{BGL}, ~ \cite{YL}, ~ \cite{L}). In these papers, the class of investigated Poisson brackets on the algebra of divided powers of characteristic 2 corresponded to standard brackets with constant coefficients for Hamiltonian Lie superalgebras containing even and odd variables.

In this paper, we develop a general theory of non-alternating Hamiltonian Lie algebras over a field of characteristic 2, which makes it possible to assign these algebras to Lie algebras of Cartan type. Here, a classification of graded non-alternating Hamiltonian Lie algebras is obtained based on the constructed complete system of invariants of non-alternating symmetric bilinear forms of characteristic 2 with respect to the flag automorphism group (Theorem 1.1). Namely, two graded non-alternating Hamiltonian Lie algebras over the algebra of divided powers defined by the flag $ \mathscr{F} $ are isomorphic if and only if the invariants of the corresponding bilinear forms coincide (Theorem 5.2). In contrast to the classical Hamiltonian case, there is a large number of equivalence classes of non-alternating forms with constant coefficients for a fixed nontrivial flag, which correspond to new simple graded Lie algebras. It is shown that the filtered deformations of graded non-alternating Hamiltonian Lie algebras with certain exceptions correspond to non-alternating Hamiltonian forms with polynomial coefficients in divided powers (Theorem 2.3). Note that classical graded Hamiltonian Lie algebras can have filtered deformations which correspond to differential forms with non-polynomial coefficients ~ \cite {Kac}. Moreover, it is proved that, in contrast to the classical Hamiltonian algebras, under fairly general conditions, for example, when the heights of some variables are greater than 1, the graded non-alternating Hamiltonian algebras are rigid with respect to filtered deformations (Theorem 3.2). A description of derivations and automorphisms of filtered (and in particular, graded) non-alternating Hamiltonian Lie algebras is given (Theorems 3.2, ~5.2) All results are valid except for some cases when the heights of some variables are 1, or when the number of variables is $ n = 2, ~ 3, $ or $ 4. $ The fact that these exceptions are not random is confirmed by the examples of simple filtered Lie algebras announced at the end of the paper. It has been shown earlier that a graded non-alternating Hamiltonian Lie algebra of four variables of height 1 is isomorphic to a classical filtered Hamiltonian Lie algebra (\cite{KKCh1}). In Section 0 the invariant definition of a complex of symmetric differential forms in divided powers, a description of its cohomology, other necessary definitions, and preliminary results are given.

 We employ the technique used in the study of modular Lie algebras of Cartan type (see ~ \cite {St}).
The work of V.G. Kac ~ \cite{Kac}, in which it was shown that all filtered deformations of graded Lie algebras of Cartan type correspond to more general forms of the same kind, exerted a significant impact on the field. In this work the results and methods of the general theory of filtered Lie algebras (Spencer cohomology, Serre's involutivity criterion, etc.) were used for the first time in the case of characteristic $p$. S.M. Skryabin in ~ \cite{S1} --  ~\cite{Sk2} developed another approach to the description of filtered deformations of Lie algebras of Cartan type, based on the theory of coinduced modules ~ \cite{Blattner}, ~ \cite{K} . Both approaches can be applied to the study of filtered deformations of non-alternating Hamiltonian Lie algebras. In the present paper the results and methods of S.M. Skryabin are used. Earlier, invariants of classical Hamiltonian forms with coefficients in divided power algebras (so called polynomial case) were found that determine the isomorphism class of general Hamiltonian Lie algebras, the value of the form at zero and its cohomological class ~\cite{Kac},~ \cite{KuKir}. \cite{BG}. The problem of classification of classical Hamiltonian forms was considered in ~\cite{BG} and  was completely solved by S.M. Skryabin (see ~\cite{S1}, ~\cite{S}). We use the methods of S.M. Skryabin \cite{S} in the study of non-alternating Hamiltonian forms and, in particular, in the construction of invariants of non-alternating symmetric forms of characteristic 2.

The ground field $ K $ is assumed to be a perfect one of characteristics 2.

\section{Symmetric differential forms}

Let $R$ be a commutative algebra over a field $K$ of characteristic $p>0$, $E$ be a free $R$-module of rank $n$. The commutative $R$-algebra of divided powers $O (E)$ is given by generators $u^{(r)}$, $u\in E$, $r\in\mathbb{Z}$, $r\geqslant 0$ and relations
\begin{equation}\label{eqeq 0.1}
\begin{array}{c}
	u^{(0)}=1, ~~u^{(1)}=u, \\
	(au)^{(r)}=a^{r}u^{(r)}, ~~a\in K, \\
	u^{(r)} u^{(s)} = \dbinom{r+s}{r}u^{(r+s)}, \\
	(u_{1}+u_{2})^{(r)} = \sum\limits_{i=1}^{r} u_{1}^{(i)} u_{2}^{(r-i)}
\end{array}\end{equation}
(see \cite{KSh}).

Let $\{ x_{1}, \ldots, x_{n}\} $ be the basis of $E$ over $R$. Elements $x^{(\alpha)} = x_{1}^{(\alpha_{1})}\cdot \ldots\cdot x_{n}^{(\alpha_{n})}$, $\alpha = (\alpha_{1}, \ldots, \alpha_{n})$, $\alpha_{i}\in\mathbb{Z}$, $\alpha_{i} \geqslant 0$ form the the monomial basis of $O (E)$,
\begin{equation}\label{eqeq 0.2}
	x^{(\alpha)}\cdot x^{(\beta)} = \dbinom{\alpha+\beta}{\alpha}x^{(\alpha+\beta)}.
\end{equation}

An algebra $O (E)$ with a fixed basis $\{ x^{(\alpha)} \}$ is denoted by $O(n)$. The algebra $O(E)$  has the standard grading,  $O (E) = \bigoplus\limits_{i\geqslant 0} O_{i}$, $O_{i}= \langle x^{(\alpha)} ~\big|~ |\alpha|=\alpha_{1}+ \ldots+ \alpha_{n}= i \rangle$, and filtration $O(E)_{(i)} = \bigoplus\limits_{j\geqslant i} O_{j}$. The standard grading and filtration are independent on the $\{ x_{i} \}$ basis of the module $E$. $O(E)$ is a local algebra with maximal ideal $\mathfrak{m} = O(E)_{(1)}$. The completion of $O(E)$ by the system of ideals $\{ \mathfrak{m}^{(i)} = O(E)_{(i)} \}$ is denoted by $\widehat{O}(E)$.

Let $V$ be a free $R$-module of rank $n$, $E=\text{Hom}_{R} (V,R)$. The algebra $O(E)$ is canonically isomorphic to the algebra of symmetric polynomial functions on $V$ with values in $R$ with multiplication
\begin{equation}\label{eqeq 0.3}
	\omega_{l}\omega_{r} (v_{1}, \ldots, v_{l+r}) = \sum_ {\sigma} \omega_{l}(v_{\sigma(1)}, \ldots, v_{\sigma(l)})\omega_{r} (v_ {\sigma(l+1)}, \ldots, v_ {\sigma(l+r)}),
\end{equation}
where the sum is taken over all substitutions of $\sigma\in S_{l+r}$ such that $\sigma(1) < \ldots < \sigma(l)$, $\sigma(l+1) < \ldots < \sigma(l+r)$.
Indeed, the symmetric polynomial map $\omega_{r}\colon V^{r}\rightarrow R$ corresponds to $\widehat {\omega}_{r}\in S^{r}(V)^{\ast} = \text{Hom}_{R}(S^{r} (V), R)$,
\begin{equation}\label{eqeq 0.4}
	\widehat{\omega}_{r} (v_{1}\cdot \ldots \cdot v_{r}) = \omega_{r} (v_{1}, \ldots, v_{r}).
\end{equation}
Here $S(V)$ is the symmetric bialgebra of $R$-module $V$, $S(V)=\bigoplus\limits_{r\geqslant 0} S^{r} (V)$, with comultiplication $\Delta$, $\Delta(v)= v\otimes 1 + 1\otimes v$, and the counit $\varepsilon $, $\varepsilon(v) = 0$, $v\in V$. Correspondence \eqref{eqeq 0.4} is an isomorphism of algebras $S(V)^{\ast}=\bigoplus\limits_{r\geqslant 0} S^{r}(V)^{\ast}$ with the multiplication $\widehat{\omega}_{l}\widehat{\omega}_{r}= \widehat{\omega}_{l}\otimes\widehat{\omega}_{r} \circ \Delta$ and the algebra of symmetric multilinear functions on $V$ with multiplication \eqref{eqeq 0.3}.

For $u\in S^{1}(V)^{\ast} = E = \text{Hom}_{R}(V,R)$ the divided powers $u^{(r)}$, $r\geqslant 0$ are defined as follows 
$$u^{(0)}=1, ~~u^{(r)} (v_{1}\cdot \ldots\cdot v_{r})= u (v_{1})\cdots u (v_{r}),$$
for which \eqref{eqeq 0.1} relations \eqref{eqeq 0.1} hold. Let $\{ v_{i} \} $ be the basis of $V$, $ \{ w_{i} \} $ be the dual basis of the module $E$. Denote by $\{ w^{(\alpha)} \}$ the basis of $S(V)^{\ast}$, the dual to the basis $\{ v^{\alpha}, ~\alpha=(\alpha_{1}, \ldots, \alpha_{n}),  \alpha_{i}\geqslant 0 \}$ of $S(V)$. Then
$$w^{(\alpha)}w^{(\beta)} = \dbinom{\alpha+\beta}{\alpha}w^{(\alpha+\beta)}.$$
Thus, $S(V)^{\ast}$ isomorphic to the algebra $O (E)$. In particular $u^{p}=0$ for any 1-form $u\in V$.

Divided powers can be defined for any $u\in \mathfrak{m}\subset O(E)$ (or $u\in \widehat{\mathfrak{m}}\subset \widehat{O}(E)$) so that the \eqref{eqeq 0.1} and
$$(u^{(l)})^{(r)}= \dfrac{(lr)!} {(l!) ^{r}r!}u^{(lr)}, ~~(\omega_{1}\omega_{2})^{(r)}= \omega_{1}^{r}\omega_{2}^{(r)},$$
are fulfilled (see \cite{Kac}, \cite{Sk1}, \cite{St}).

Using the isomorphism $\omega_{r}\mapsto \widehat{\omega}_{r}$ (see \eqref{eqeq 0.4}), we will identify the algebra of symmetric multilinear functions on $V$ and the algebra of divided powers $O(E)= S(V)^{\ast}$.

Let $R=K$. A derivation $ D $ of algebra $O(E)$ is called \textit{special} if $Du^{(r)}= u^{(r-1)}Du$, $u\in E$. The Lie algebra of all special derivations of the algebra $O(E)$ is denoted by $W(E)$.

Let $\mathscr{F}\colon E=E_{0}\supseteq E_{1}\supseteq \ldots \supseteq E_{r}\supset E_{R+1}=\{0\}$ be a flag of $E$. A basis $\{ x_{i} \}$ of $E$ is called \textit{coordinated with} $\mathscr{F}$  if $\{ x_{i} \}\cap E_{j}$ is a basis of  $E_{j}$. A subalgebra $O(\mathscr{F}) $ of $ O(E) $ is generated by $ u^{(p^{j})}$, $ u\in E_{j}$,$j= 0, \ldots, r$.

Let $\{ x_{i} \} $ be a basis of $E$ coordinated with $\mathscr{F}$. The elements $x^{(\alpha)}$, $\alpha= (\alpha_{1}, \ldots, \alpha_{n})$, $0\leqslant\alpha_{i}< p^{m_{i}}$ form the basis of $O (\mathscr{F})$. Here $m_{i}$ is the height of variable $x_{i}$, $m_{i} = \min\{ s ~|~ x_{i}\notin E_{s} \}$. The algebra $O (\mathscr{F})$ with such basis is denoted by $O(n, \overline{m})$, $\overline{m}=(m_{1}, \ldots, m_{n})$. By $W (\mathscr{F})$ (resp. $W(n, \overline{m})$) is denoted the Lie subalgebra in $W (E)$ consisting of all special derivations preserving  $O(\mathscr{F})$ (resp. $O(n, \overline{m})$). The partial derivations $\partial_{i}$, $\partial_{i}x^{(\alpha)}=x^{(\alpha - \varepsilon_{i})}$, $\varepsilon_{i}= (0, \ldots, 1, \ldots, 0)$, form the basis of the free $O(\mathscr{F})$-module $W(\mathscr{F})$. The standard grading (resp. filtration) of the algebra $O (E)$ induces a standard grading (resp. filtration) of algebras $W(E)$ and $W (\mathscr{F})$, $W (\mathscr{F})_{i}= \langle x^{(\alpha)}\partial_{j}, ~ / \ alpha|= i+1, ~j=1, \ldots, n \rangle$, $x^{(\alpha)} \ in O (\mathscr{F}$, $W (\mathscr{F})_{(i)}= \langle x^{(\alpha)}\partial_{j}, ~|\alpha|\geqslant i+1, ~j=1, \ldots, n \rangle$ (similar for $W (E)$). Let $\mathfrak{m}$ be the maximum ideal of $O (E)$, then $\mathfrak{m}\cap O (\mathscr{F})$ be the maximum ideal of $O (\mathscr{F})$, which is denoted by $\mathfrak{m} (\mathscr{F})$.

Now let $R= O(\mathscr{F})$, $V = W (\mathscr{F})$. Then $V^{\ast}= \text{Hom}_{R} (W (\mathscr{F}),R)$ is $R$-module of differential $1$-forms and is denoted by $S\Omega^{1}(\mathscr{F})$. The $R$-algebra of the divided powers $O(V^{\ast}) = S(V)^{\ast}$ is the algebra of symmetric differential forms denoted by $S\Omega(\mathscr{F}) = \bigoplus\limits_{i\geqslant 0} S\Omega^{i}(\mathscr{F})$, $S\Omega^{0} = R= O(\mathscr{F})$, $S\Omega^{1}(\mathscr{F}) = \Omega^{1}(\mathscr{F}) = V^{\ast}$.

In what follows we assume that $K$ is a perfect field of characteristic $ 2 $.

The inner product of $D\in W(\mathscr{F})$, $$D\lrcorner \omega_{r} (D_{1}, \ldots, D_{r-1}) = \omega_{r} (D, D_{1}, \ldots, D_{r-1})$$ has properties
$$D\lrcorner (\omega_{1}\omega_{2})= (D\lrcorner \omega_{1})\omega_{2} + \omega_{1} (D\lrcorner \omega_{2}),$$
$$D\lrcorner \omega^{(r)}= (D\lrcorner \omega)\omega^{(r-1)}.$$

The Lie algebra $W (\mathscr{F})$ acts naturally on $S\Omega (\mathscr{F})$ by derivations, i.e.
$$D (f\omega)= D(f)\omega + f (D\omega),$$
$\omega\in s\Omega(\mathscr{F})$, $f\in O(\mathscr{F})$.

The following identity, known for skew-symmetric forms, is also true in $S\Omega(\mathscr{F})$
\begin{equation}\label{eqeq 0.5}
	D_{1} (D_{2}\lrcorner \omega)= [D_{1}, D_{2}]\lrcorner \omega + D_{2}\lrcorner (D_{1}\omega).
\end{equation}

Over a field of characteristic $2$, we define the differential $d\colon s\Omega^{r}(\mathscr{F}) \rightarrow S\Omega^{r+1} (\mathscr{F})$ such that $df(D)= D(f)$,
\begin{gather*}
	(d\omega_{r}) (D_{1},\ldots,D_{r+1})=\sum_{i=1}^{r+1}D_{i}(\omega(D_{1},\ldots, \widehat{D_{i}}, \ldots, D_{r+1})+ \\
	+\sum_{i<j}\omega([D_{i},D_{j}], D_{1},\ldots, \widehat{D_{i}},\ldots, \widehat{D_{j}}, \ldots,D_{r+1}),
\end{gather*} \ \ [- 1cm]
$$d^{2}=0, ~d\omega^{(r)}= \omega^{(r-1)} d\omega, ~d(\omega_{1}\omega_{2})= (d\omega_{1})\omega_{2}+ \omega_{1}(d\omega_{2}).$$

We consider alternating forms as symmetric, such that $ \omega_{r} (D_{1},\ldots, D_{r}) = 0$ if $D_{i}=D_{j}$ for some $i\neq j$. Therefore, the de Rham complex $(\Omega (\mathscr{F}), d)$ is a subcomplex of the complex $(S\Omega (\mathscr{F}), d)$. The Cartan homotopy formula also holds in $S\Omega (\mathscr{F})$:
\begin{equation}\label{eqeq 0.6}
	D\omega= D\lrcorner d\omega+ d(D \lrcorner \ omega).
\end{equation}
\begin{theorem}[\cite{KKCh}]\label{ter 0.1}
{\it Let $x_{i}$, $i= 1,\ldots, n$ be the standard variables of the algebra of divided powers $O (\mathscr{F})= O (n, \overline{m})$ over field $ K $ of characteristics $2$, $\overline{x_{i}}= x_{i}^{(2^{m_{i} -1)}}$, $B (n)$ graded algebra of divided powers over $K$ in variables $(dx_{i})^{(2)}$, $i= 1,\ldots, n$ of degree two, $(\Omega(\mathscr{F}), d)$ de Rham complex over $O(\mathscr{F})$. The following statements are true:
	\begin{itemize}
		\item[$i.~$] The cohomology ring $H^{\ast} (S\Omega (\mathscr{F}))$ is a tensor product of graded algebras 
		$$H^{\ast} (S\Omega (\mathscr{F}))= B (n)\otimes_{K} H^{\ast} (\Omega (\mathscr{F})),$$  
		\item[$ii.$] $\dim H^{i} (S\Omega(\mathscr{F}))= \dbinom{n+i-1}{i}$,
		\item[$iii.$] $H^{1}(S\Omega(\mathscr{F}))= H^{1}(\Omega(\mathscr{F})) = 
		\langle [\overline{x_{i}}dx_{i}], ~i= 1,\ldots , n \rangle $,
		\item[$iv.$] $H^{2} (S\Omega(\mathscr{F}))= \langle [(dx_{i})^{(2)}], ~i= 1,\ldots , n, ~~[\overline{x_{i}}\overline{x_{j}}dx_{i}dx_{j}], ~1\leqslant i<j\leqslant n \rangle $.
	\end{itemize}}
\end{theorem}

\begin{remark}
	1. De Rham cohomology over the algebra $O (\mathscr{F})$ is obtained in ~\cite{Kr} (see also ~\cite{St}). \\
	2. All definitions of this section hold for the case of infinite dimensional algebras such as $O(E)$, $\widehat{O}(E)$, $O(\mathscr{F})$, $\widehat{O}(\mathscr{F})$ for the infinite flag $\mathscr{F}$, with natural constraints of topological or other nature. For example, for $\widehat{O}(\mathscr{F})$ it is natural to consider continuous in $\{mathfrak{m}^{(i)}\}$-adic topology special derivations, and for algebras $O (E)$, $\widehat{O}(E)$ in Theorem~\ref{ter 0.1} we must take into account that $H^{\ast} (\Omega)= H^{0}(\Omega) = K$ (Poincar\' Lemma).
\end{remark}

The continuous isomorphism $\sigma\colon \widehat{O}{E})\rightarrow \widehat{O}(E') $ is called admissible if $\sigma (x^{(\alpha)})= \sigma (x)^{(\alpha)}.$ If $\sigma (O(\mathscr{F})=O(\mathscr{F'})$, then $\sigma $ is called an admissible isomorphism of the corresponding algebras. In this case, $\sigma $ induces an isomorphism of Lie algebras $W (\mathscr{F})$ and $W (\mathscr{F}')$,
$\sigma (D)= \sigma\circ D\circ \sigma^{-1}$ (~\cite{W}). An admissible isomorphism $\sigma$  continues to the isomorphism 
$S\Omega(\mathscr{F})\rightarrow S\Omega(\mathscr{F}')$ by the rule

$$(\sigma\omega) (D_{1},\ldots, D_{k})= \sigma (\omega (\sigma^{-1}D_{1}\sigma, \ldots, \sigma^{-1}D_{k}\sigma)).$$

Moreover, $\sigma d\omega = d\sigma\omega $, $\sigma(\omega_{1}\omega_{2}) = \sigma\omega_{1}\cdot \sigma\omega_{2}$ and $\sigma(\omega^{(k)})= (\sigma\omega)^{(k)}$.

  The following Lemma will be applied in section 3.
\begin{lemma}\label{lemma 0.2}
{\it Let $\omega= \sum\limits_{i=1}^{n}\omega_{ii} (dx_{i})^{(2)} + \sum\limits_{i<j}\omega_{ij}dx_{i}dx_{j}$ be a symmetric $2$ - form with coefficients from $\widehat{O} (E)$. If $d\omega= 0$, then $\omega_{ii}\in K$.}
\end{lemma}
\begin{proof}
	Write $\omega$ as $\omega = \omega_{1} + \omega_{2}$, $\omega_{1}= \sum\limits_{i=1}^{n}\omega_{ii}(dx_{i})^{(2)}$, $\omega_{2}= \sum\limits_{i<j}\omega_{ij}dx_{i}dx_{j}$. Obviously $\omega_{2}\in \widehat{\Omega}(E)$. Let $I$ be the ideal of $\widehat{S\Omega}(E)$ generated by $(dx_{i})^{(2)}$, $i= 1,\ldots, n$. Here $\widehat{S\Omega}(E)$ is the algebra of $\widehat{O}(E)$-multilinear maps from $\widehat{W}(E)$ to $\widehat{O}(E)$. Since $d\omega_{2}\in \widehat{\Omega}(E)$, $d(f(dx_{i})^{(2)}) = \\ =\sum\limits_{j} \partial_{j}f dx_{j}dx_{i}^{(2)}\in I$ and $I\cap \widehat{\Omega}(E)=0$, then $d\omega_{1}= d\omega_{2}=0$. Now, $d\omega_{1} = \sum\limits_{i} d\omega_{ii} (dx_{i})^{(2)} =0$, therefore, $d\omega_{ii}=0$.
\end{proof}

 Symmetric $2$-form $\omega\in S\Omega^{2}(\mathscr{F})$ is called \textit{non-alternate} if there exists $D\in W(\mathscr{F})$ such that $\omega(D,D)\neq 0$. Let
$$\omega= \sum_{i=1}^{n}\omega_{ii} (dx_{i})^{(2)} + \sum_{i<j}\omega_{ij}dx_{i}dx_{j}$$
be a non-alternating $2$ - form. For $j>i$ put $\omega_{ji}= \omega_{ij}$, $M=(\omega_{ij})$ is the matrix of $\omega$. The form $\omega$ is called \textit{nondegenerate} if $\det\omega = \det M$ is invertible in $O(\mathscr{F})$.
Let $a_{ij}= \omega_{ij}(0) = \omega_{ij}(mod \mathfrak{m})\in O(\mathscr{F})\big/\mathfrak{m}= K$, $\omega(0)= \sum\limits_{i=1}^{n}a_{ii}(dx_{i})^{(2)} + \sum\limits_{i<j}a_{ij}dx_{i}dx_{j}$. Obviously, the form $\omega$ is nondegenerate if and only if $\det\omega(0) = \det M(0)\neq 0$. Throughout what follows, we will denote $(\overline{\omega}_{ij})= M^{-1}$, $(\overline{a}_{ij})= M^{-1} (0)$.

The closed nondegenerate non-alternating form $\omega\in s\Omega^{2}(\mathscr{F})$ is called the \textit{non-alternating Hamiltonian} form. The corresponding Lie algebra of Hamiltonian vector fields is denoted by $\widetilde{P}(\mathscr{F}, \omega)$,
$$\widetilde{P} (\mathscr{F}, \omega)= \{ D\in W(\mathscr{F}) ~|~ D\omega=0 \}.$$

It follows from \eqref{eqeq 0.6} that ~$D\in \widetilde{P}(\mathscr{F}, \omega)$~ if and only if ~$d(D\lrcorner \omega)=0$,~ i.e. ~$D\lrcorner \omega\in Z^{1}(S\Omega^{1}(\mathscr{F})) = Z^{1}(\Omega (mathscr{F}))$. By Theorem~\ref{ter 0.1} $D\lrcorner \omega= df$, $f\in O(\mathscr{F})+ \langle x_{i}^{(2^{m_{i}})}, ~i= 1,\ldots, n \rangle$. Whence 
$$D= D_{f}=\sum_{i, j}\overline{\omega}_{ij} \partial_{j}f\partial_{i}.$$
From the formula \eqref{eqeq 0.5} we obtain that
$$[D_{f}, D_{g}]\lrcorner \omega= D_{f}(D_{g}\lrcorner \omega) = D_{f}dg = d(D_{f}(g)).$$

Thus, the correspondence $f\mapsto D_{f}$ is an isomorphism of the Lie algebra $\widetilde{P}(\mathscr{F}, \omega)$ and the Lie algebra $\widetilde{O}(\mathscr{F})\big/ K$ with Poisson bracket
\begin{equation}\label{eqeq 0.7}
	\{ f, g \} = D_{f} (g) = D_{g} (f)= \sum_{i,j}\overline{\omega}_{ij} \partial_{i}f\partial_{j}g.
\end{equation}

Let $$P(\mathscr{F}, \omega)=\{ D_{f} ~|~ f\in O(\mathscr{F})\big/ K \}.$$

It follows from \eqref{eqeq 0.7} that $P (\mathscr{F}, \omega)$ is the ideal of $\widetilde{P} (\mathscr{F}, \omega)$ of codimension $n$. A Lie algebra $\mathscr{L}$ such that $P^{(1)} (\mathscr{F}, \omega)\subseteq  \mathscr{L} \subseteq \widetilde{P}(\mathscr{F}, \omega)$ will be called the non-alternating Hamiltonian Lie algebra.

\begin{remark}
	Poincar\'e algebra $\widehat{O} (\mathscr{F})$ (resp. $O (\mathscr{F})$) with Poisson bracket \eqref{eqeq 0.7} is a Leibniz algebra with center $K$, but not a Lie algebra.
\end{remark}

\section{Canonical form of bilinear non-alternating \\ symmetric forms with respect to the flag}

Let $V$ be a finite-dimensional vector space over the field $K$ of an even characteristic. $b$~ bilinear nondegenerate non-alternating symmetric form on $V$, $V^{0}$ the hyperplane consisting of isotropic vectors (i.e. $V^{0}=\{ v\in V ~|~ b(v,v)=0\} $) and
$\mathscr{F}\colon 0=V_{0}\subseteq V_{1}\subseteq \ldots$  be a flag of $V$ such that $V_{q}=V$ for sufficiently large $q$. For the subspace $L\subseteq V$, denote by $L^{\perp}$ its orthogonal complement with respect to $b$. We recall canonical isomorphisms and pairings constructed in 
~\cite{S}. Let
$$\Phi_{q}\mathscr{F}=V_{q}/V_{q-1} \text{ for } q\geqslant 1,$$
$$\Phi_{q} (\mathscr{F})_{r}=(V_{q}\cap V_{r}^{\perp}+V_{q-1}) {\big /}V_{q-1} \text{ for } q\geqslant 1, r\geqslant 0,$$
$$\Phi_{q}\Phi_{r}\mathscr{F}=\Phi_{q} (\mathscr{F})_{r-1}{\big /}\Phi_{q}(\mathscr{F})_{r} \text{ for } q,r\geqslant 1.$$

There are canonical isomorphisms
\begin{equation*}
\Phi_{q}\Phi_{r}\mathscr{F}\cong (V_{q}\cap V_{r-1}^{\perp}+V_{q-1}){\big/} (V_{q}\cap v_{r}^{\perp}+V_{q-1})\cong 
\end{equation*}
\begin{equation}\label{equ:mul1} 	
\cong (V_{q}\cap V_{r-1}^{\perp}){\big/} (V_{q}\cap V_{r}^{\perp}+V_{q-1}\cap V_{r-1}^{\perp})=\Phi_{qr} 
\end{equation}

The form $b$ induces a non-degenerate pairing of subspaces $U, W\subset V$.
$$U{\big/} (U\cap W^{\bot})\times W{\big/} (U^{\bot}\cap W)\rightarrow K, ~~(u, w)\mapsto b(u, w).$$

Let $L_{1},M_{1}\subseteq V$ and $L_{0}\subseteq L_{1}$, $M_{0}\subseteq M_{1}$ be subspaces of $V$. We construct a pairing $(L_{1}\cap M_{0}^{\bot})\times (L_{0}^{\bot}\cap M_{1})\rightarrow K$. Find the left and right core.
Since $(L\cap M^{\bot})^{\bot}=L^{\bot}+M$, 
\begin{multline*} 
(L_{1}\cap M_{0}^{\bot})^{\bot}\cap (L_{0}^{\bot}\cap M_{1})=(L_{1}^{\bot}+M_{0})\cap (L_{0}^{\bot}\cap M_{1})= \\
=(L_{1}^{\bot}\cap M_{1}+M_{0})\cap L_{0}^{\bot}=L_{1}^{\bot}\cap m_{1}+L_{0}^{\bot}\cap M_{0}
\end{multline*}
and analogously 
$$(L_{1}\cap M_{0}^{\bot})\cap(L_{0}^{\bot}\cap M_{1})^{\bot}=L_{1}\cap M_{1}^{\bot}+L_{0}\cap M_{0}^{\bot}.$$

Hence, $b$ induces a non-degenerate bilinear pairing
$$(L_{1}\cap M_{0}^{\bot}){\big /}(L_{1}\cap M_{1}^{\bot}+L_{0}\cap M_{0}^{\bot})\times(L_{0}^{\bot}\cap M_{1}){\big /}(L_{1}^{\bot}\cap M_{1}+L_{0}^{\bot}\cap M_{0})\rightarrow K.$$

Given isomorphisms of \eqref{equ:mul1}, we obtain non-degenerate bilinear pairings
\begin{equation}\label{equ:spar2}
	\Phi_{qr}\times \Phi_{rq}\rightarrow K, \quad q, r\geqslant 1. 
\end{equation}

 For $q, r\geqslant 1$ define
\begin{center}
	$\Phi_{qr}^{0}$ as the image of $V_{q}\cap V_{r-1}^{\perp}\cap V^{0}$ under the canonical projection on $\Phi_{qr}$, 
	$$ n_{qr}=\dim\Phi_{qr}, ~~~n_{qr}^{1}=\dim\Phi_{qr} - \dim\Phi_{qr}^{0}.$$
\end{center}

Note that $n_{qr}^{1}$ can only be equal to $0$ or $1$.

The decomposition of $V$ into the direct sum of its subspaces $P$ and $Q$ is called coordinated with the flag $\mathscr{F}$ if $V_{j}=V_{j}\cap P+V_{j}\cap Q$ for all $j\geqslant 0$. The number $m= \min\{ j ~ / ~ e\in V_{j}\}$ is the height of an element $e$.

\begin{lemma}[cf. \cite{S}]\label{lem:lem5.1}
{\it Suppose that for some $q, r\geqslant 1$ either $q\neq r$, or $q=r$, $n_{qq}^{1}=0$. If  $u\in V_{q}\cap V_{r-1}^{\bot}$, $v\in V_{r}\cap V_{q-1}^{\bot}$, and $b(u,v)=1$ then the decomposition
	$$V=\langle u, v \rangle \oplus \langle u,v \rangle^{\bot}$$
	coordinated with $\mathscr{F}$ and $b(u,u)\cdot b (v,v)=0$.}
\end{lemma}
\begin{proof}
	(1) Assume for certainty that $q<r$. 
	
	(1.1) If $j<q$, then $j\leqslant q-1 < r-1$ and $V_{j}\subseteq V_{q-1}$, $V_{j}\subseteq V_{r-1}$. Then $V_{j}\subseteq u^{\bot}$ and $V_{j}\subseteq v^{\bot}$, hence $V_{j}\subseteq \langle u, v \rangle^{\bot}$. 
	
	(1.2) If $j\geqslant r$, then $V_{q},V_{r}\subseteq V_{j}$, then $\langle u,v \rangle \subseteq V_{j}$ and  $V_{j}=V_{j}\cap V=V_{j}\cap(\langle u,v \rangle + \\ + \langle u,v \rangle^{\bot})=\langle u,v \rangle + \langle u,v \rangle^{\bot}\cap V_{j}.$
(1.3) Let $q\leqslant j<r$. Then $V_{q}\subseteq V_{j}\subseteq V_{r-1}$, hence $u\in V_{j}\cap V_{j}^{\bot}$, i.e. $b(u, u)=0$. Since $b(u,v)=1$, then $v\notin V_{j}$ and for any $u'\in V_{j}$ such that $b (u', v)=\alpha$ we have $b (\alpha u+u', v)=0$. Let's denote $w=\alpha u+u'$ and get the decomposition $u ' =\alpha u+w$, i.e. 
	$$V_{j}=\langle u \rangle + \{ w\in V_{j} ~|~ b(w,v)=0\}=\langle u,v \rangle\cap V_{j} + \langle u,v \rangle ^{\bot}\cap V_{j}.$$
	
	(2) for $q=r$ we should consider only cases similar to (1.1) and (1.2). Since $n_{qq}^{1}=0$, the set $V_{q}\cap V_{q-1}^{\bot}$ does not contain a vector of nonzero length and therefore for $u, v \in v_{q}\cap V_{q-1}^{\bot}$, $b(u,u)=0$ and $b(v, v)=0$.
\end{proof}

\begin{lemma}\label{lem:lem5.2}
{\it Let $u\in V_{q}\cap V_{q-1}^{\bot}$ be given For some $q\geqslant 1$, with $b (u,u)=1$. Then decomposition
	$$V=\langle u \rangle \oplus \langle u \rangle^{\bot}$$
	coordinated with $\mathscr{F}$.}
\end{lemma}
\begin{proof}
	If $j<q$, then $V_{j}\subseteq V_{q-1}$ and $V_{j}\subseteq u^{\bot}$. 	
	If $j\geqslant r$, then $V_{q}\subseteq V_{j}$ and $\langle u \rangle \subseteq V_{j}$. Now, $v_{j}=V_{j}\cap V=V_{j}\cap(\langle u \rangle + \langle u \rangle^{\bot})=\langle u \rangle + \langle u \ rangle^{\bot}\cap V_{j}$
\end{proof}

Consider triples $(V, \mathscr{F}, b)$ and call $(V, \mathscr{F}, b)$ and $(V', \mathscr{F}', b')$ \textit{equivalent} if there is an isomorphism $V\rightarrow V'$ that translates the flag $\mathscr{F}$ onto $\mathscr{F}'$, and the form $b$ into $b'$. Such isomorphism induces isomorphisms $\Phi_{qr}\rightarrow \Phi'_{qr}$ for all $q, r\geqslant 1$. Therefore, the values $n_{qr}$ and $n_{qr}^{1}$ are invariants of the triple $(V, \mathscr{F}, b)$.

\begin{theorem}\label{lem5.2}
{\it (1) There is coordinated with the flag $\mathscr{F}$ a basis of the space $V$ relative to which the matrix of non-alternatig form $b$ has the form $diag(M_{0},\ldots, M_{0}, M_{1} ,\ldots, M_{1}, 1_{s})$, where $s=\sum\limits_{q=1}^{n} n_{qq}n_{qq}^{1}$, the number of matrices $M_{1}$ is $\sum\limits_{q<r}^{n} n_{qr}n_{rq}^{1}$,
	$$M_{0}=\begin{pmatrix} 0 & 1 \\ 1 & 0\end{pmatrix} \text{ and } M_{1}=\begin{pmatrix} 0 & 1 \\ 1 & 1\end{pmatrix}.$$
	(2) Triples $(V, \mathscr{F}, b)$ and $(V', \mathscr{F}', b')$ are equivalent if and only if they have the same invariants $n_{qr}$ and $n_{qr}^{1}$ for all $q, r\geqslant 1$.}
\end{theorem}
\begin{proof}
	(1) Since the form is nondegenerate, there is $n_{qr}>0$ for some $1\leqslant q\leqslant r$. 
	
    \underline{Step 1.} Among $n_{qr}>0$ choose the one for which $q<r$ or $q=r$, $n_{qq}^{1}=0$. Due to the nondegeneracy of the pairing \eqref{equ:spar2}, there are $u, v$ that satisfy the conditions of the Lemma~\ref{lem:lem5.1}. In this case, if $n_{qr}>2$ and $~n_{rq}^{1}=1$, then take $v$ such that $b(v, v)=0$. We include elements $u,v$ in the basis. Let's put $W=\langle u, v \rangle^{\bot}$ and define the flag $\mathscr{H}$ of subspaces in $W$, assuming $W_{j}=W\cap V_{j}$ for $j\geqslant 0$. Repeat Step 1 for $W$, $H$, and ${b}\vert_{W}$ until $n_{qr}>0$ meet the condition. 
	
	\underline{Step 2.} For the remaining $n_{qq}>0$, $n_{qq}^{1}=1$ is fulfilled. Due to the nondegeneracy of the pairing \eqref{equ:spar2}, there is $u$ that satisfies the conditions of the Lemma~\ref{lem:lem5.2}. Include in the basis element $u$. Let's put $W=\langle u \rangle^{\bot}$ and define the flag $\mathscr{H}$ of subspaces in $W$, assuming $W_{j}=W\cap V_{j}$ for $j\geqslant 0$. Repeat Step 2 for $W$, $\mathscr{H}$, and $b\vert_{W}$. 
	
	\underline{Step 3.} Enumerate the basis so that the shape matrix takes the form specified in the condition.
	
	(2) Assume that the listed invariants of two triples coincide. Then the dimensions of the spaces $V$, $V'$ are equal and $n=\sum n_{qr}$. We have bases $\{e_{1},\ldots, e_{n}\}$ in $V$ and $\{e'_{1},\ldots, e'_{n}\}$ in $V'$, as in (1). Let $m_{i}$ ($m'_{i}$, respectively) be the height of the element $e_{i}$ ($e'_{i}$, respectively). Let's put for $q\neq r\geqslant1$ and $2t=n-\sum\limits_{q=1}^{n} n_{qq}n_{qq}^{1}$ 
	$$ \tilde{i}=\begin{cases} i+1, \text{ if } i=2k-1 \ \ i - 1, \text{ if } i=2k \end{cases}$$
	$$I_{qr}=\{ i ~ / ~ 1\leqslant i\leqslant 2t, ~m_{i}=q, ~m_{\tilde{i}}=r \}, $$
	$$I'_{qr}=\{ i ~ / ~ 1\leqslant i\leqslant 2t, ~m'_{i}=q, ~m'_{\tilde{i}}=r \}, $$
	$$I_{qq}=\{ i ~ / ~ 2t+1\leqslant i\leqslant n, ~~m_{i}=q\}, $$
	$$I'_{qq}=\{ i ~ / ~ 2t+1\leqslant i\leqslant n, ~~m'_{i}=q \}.$$
	
	Classes of vectors $e_{i}$, $i\in I_{qr}$ (possibly $q=r$) form the basis $\Phi_{qr}$. So $I_{qr}$ consists of $n_{qr}$ indexes and the same is true for $I'_{qr}$. Similarly, the sets $I_{qq}$ and $I'_{qq}$ consist of the same number of $n_{qq}$ indexes. Note also that $\tilde{i}\in I_{rq}$ for $i\in I_{qr}$. All of the above allows us to construct a permutation $\pi$ of indices $1, \ldots, n$, which maps each $I_{qr}$ to $I'_{qr}$, $q, r\geqslant1$, and for which $\pi \tilde{i}= \widetilde{\pi i}$ at $1\leqslant i\leqslant 2t$. Then $m'_{\pi i}= m_{i}$ for all $i$. The linear isomorphism $V\rightarrow V'$, which translates $e_{i}$ to $e '_{\pi i}$, $1\leqslant i\leqslant n$, specifies the equivalence of triples $(V, \mathscr{F}, b)$ and $(V', \mathscr{F}', b')$.
\end{proof}

\begin{remark}
\noindent 1.	From the proof of Lemma~\ref{lem:lem5.1} it follows that for the pair $(u,v)$ corresponding to the matrix $M_{1}$, the height of $u$ is less than the height of $v$.\newline
\noindent 2.  In the case where $p>2$, the canonical form of a matrix of symmetric bilinear form
contains no blocks $M_1$, and $ 1_s $ is replaced by a diagonal matrix (see ~\cite{HP}, Theorem 1, Section 2, Chapter 9).
\end{remark}

\begin{remark}
	The invariants mutually uniquely define a matrix of canonical form and a set of heights. For a triple $(V, \mathscr{F}, b)$ with invariants $n_{qr}$ and $n_{qr}^{1}$ ($q, r\geqslant 1$), the height set consists of three groups: $M_{0}$, $M_{1}$ and $1_{s}$, respectively. For $q\leqslant r$ in the first group of pair $q, r$ is included in the amount of $(1-n_{rq}^{1})n_{qr}$. For $q<r$ in the second group of the pair $q, r$ is in the amount of $n_{qr}n_{rq}^{1}$. The third group includes the height $q$ in the amount of $n_{qq}n_{qq}^{1}$.
\end{remark}

In the following examples, the differential form with constant coefficients is considered as a bilinear form on the space $V=W_{-1}$.

\textbf{Example 1.} Let $n=3$, $\mathscr{F}\colon V=V_{3}=\langle \partial_{1}, \partial_{2}, \partial_{3} \rangle \supset V_{2}=\langle \partial_{1}, \partial_{2} \rangle \supset V_{1}=\langle \partial_{1} \rangle \supset V_{0}=0$ and a coherent $\mathscr{F}$ basis form has the form $\omega=dx_{1}^{(2)} +dx_{1}dx_{3}+dx_{2}dx_{3}$. \\
We have $V^{0}= \langle \partial_{2}, \partial_{3} \rangle$ and $V_{1}^{\bot}= \langle \partial_{2}, \partial_{1}+\partial_{3} \rangle$. 
We apply the algorithm specified in the proof of Theorem~\ref{lem5.2}: \\
$\Phi_{11}=(V_{1}\cap V_{0}^{\bot}) {\big/} (V_{1}\cap V_{1}^{\bot}+ V_{0}\cap V_{0}^{\bot})= V_{1}$, ~~$V_{1}\cap V_{0}^{\perp}\cap V^{0}=0 \Rightarrow \Phi_{11}^{0}=0 \Rightarrow n_{11}^{1}=1$; \\
$\Phi_{12}=(V_{1}\cap V_{1}^{\bot}) {\big/} (V_{1}\cap V_{2}^{\bot}+ V_{0}\cap V_{1}^{\bot})= 0$; \\
$\Phi_{13}=(V_{1}\cap V_{2}^{\bot}){\big/} (V_{1}\cap V_{3}^{\bot}+ V_{0}\cap V_{2}^{\bot})= 0$; \\
$\Phi_{21}=(V_{2}\cap V_{0}^{\bot}){\big/} (V_{2}\cap V_{1}^{\bot}+ V_{1}\cap V_{0}^{\bot})= 0$; \\
$\Phi_{22}=(V_{2}\cap V_{1}^{\bot}){\big/} (V_{2}\cap V_{2}^{\bot}+ V_{1}\cap V_{1}^{\bot})= 0$; \\
$\Phi_{23}=(V_{2}\cap V_{2}^{\bot}) {\big/} (V_{2}\cap V_{3}^{\bot}+ V_{1}\cap V_{2}^{\bot})= \langle \partial_{2} \rangle$, ~$V_{2}\cap V_{2}^{\perp} \cap V^{0}=\langle \partial_{2} \rangle \Rightarrow \Phi_{23}^{0}=\Phi_{23} \Rightarrow n_{23}^{1}=0$ \\
$\Phi_{31}=(V_{3}\cap V_{0}^{\bot}){\big/} (V_{3}\cap V_{1}^{\bot}+ V_{2}\cap V_{0}^{\bot})= 0$; \\
$\Phi_{32}=(V_{3}\cap V_{1}^{\bot}){\big /}(V_{3}\cap V_{2}^{\bot}+ V_{2}\cap V_{1}^{\bot})= \langle \partial_{1}+\partial_{3} \rangle$; $V_{3}\cap V_{1}^{\perp}\cap V^{0}=\langle \partial_{2} \rangle \Rightarrow \Phi_{32}^{0}=0 \Rightarrow n_{32}^{1}=1$; \\
$\Phi_{33}=(V_{3}\cap V_{2}^{\bot}){\big/} (V_{3}\cap V_{3}^{\bot}+ V_{2}\cap V_{2}^{\bot})= 0$.\\

We see that $n_{23}>0$. Since $n_{32}^{1}=1$, the pair $(\partial_{2}, \partial_{1}+\partial_{3})$ corresponds to the matrix $M_{1}$. We include these vectors in the basis. Step 1 completed. \\
We see that $n_{11}>0$ and $n_{11}^{1}=1$. Include the vector $\partial_{1}$ in the basis.\\
As a result, $\omega=dx_{1}dx_{2}+dx_{2}^{(2)}+dx_{3}^{(2)}$. Set of heights: $(2,3,1)$.

\textbf{Example 2.} Let $\mathscr{F}\colon V=V_{4} \supset V_{3} = V_{2} = V_{1} \supset V_{0}=0$ and $V_{4} {\big /} V_{3}= \langle \partial_{1} \rangle$, in some coordinated with the flag, the form matrix has the form of $\left( \begin{array}{ccc} 1 & 0 & 0 \\ 0 & 1 & 0 \\ 0 & 0 & B \end{array} \right)$, where $B$ is a non-degenerate skew-symmetric matrix. 
Then $V^{0}= \langle \partial_{1}+ \partial_{2}, \partial_{3}, \ldots, \partial_{n} \rangle$ and \\
$\Phi_{12}=(V_{1}\cap V_{1}^{\bot}) {\big/} (V_{1}\cap V_{2}^{\bot}+ V_{0}\cap V_{1}^{\bot})= \Phi_{21}=(V_{2}\cap V_{0}^{\bot}){\big/} (V_{2}\cap V_{1}^{\bot}+ V_{1}\cap V_{0}^{\bot})= 0$; \\
$\Phi_{13}=(V_{1}\cap V_{2}^{\bot}) {\big/} (V_{1}\cap V_{3}^{\bot}+ V_{0}\cap V_{2}^{\bot})= \Phi_{31}=(V_{3}\cap V_{0}^{\bot}){\big/} (V_{3}\cap V_{1}^{\bot}+ V_{2}\cap V_{0}^{\bot})= 0$; \\
$\Phi_{14}=(V_{1}\cap V_{3}^{\bot}) {\big/} (V_{1}\cap V_{4}^{\bot}+ V_{0}\cap V_{3}^{\bot})= \Phi_{41}=(V_{4}\cap V_{0}^{\bot}) {\big/} (V_{4}\cap V_{1}^{\bot}+ V_{4}\cap V_{0}^{\bot})= 0$; \\
$\Phi_{23}=(V_{2}\cap V_{2}^{\bot}) {\big/} (V_{2}\cap V_{3}^{\bot}+ V_{1}\cap V_{2}^{\bot})= \Phi_{32}=(V_{3}\cap V_{1}^{\bot}){\big/} (V_{3}\cap V_{2}^{\bot}+ V_{2}\cap V_{1}^{\bot})= 0$; \\
$\Phi_{24}=(V_{2}\cap V_{3}^{\bot}) {\big/} (V_{2}\cap V_{4}^{\bot}+ V_{1}\cap V_{3}^{\bot})= \Phi_{42}=(V_{4}\cap V_{1}^{\bot}) {\big/} (V_{4}\cap V_{2}^{\bot}+ V_{3}\cap V_{1}^{\bot})= 0$; \\
$\Phi_{34}=(V_{3}\cap V_{3}^{\bot}) {\big/} (V_{3}\cap V_{4}^{\bot}+ V_{2}\cap V_{3}^{\bot})= \Phi_{43}=(V_{4}\cap V_{2}^{\bot}) {\big/} (V_{4}\cap V_{3}^{\bot}+ V_{3}\cap V_{2}^{\bot})= 0$; \\
$\Phi_{11}=(V_{1}\cap V_{0}^{\bot}){\big/} (V_{1}\cap V_{1}^{\bot}+ V_{0}\cap V_{0}^{\bot})= V_{1}$, ~~~$\Phi_{11}^{0}=\langle \partial_{3}, \ldots, \partial_{n} \rangle$, ~$n_{11}^{1}=1$; \\
$\Phi_{22}=(V_{2}\cap V_{1}^{\bot}) {\big/} (V_{2} \cap V_{2}^{\bot}+ V_{1}\cap V_{1}^{\bot})= \Phi_{33}=(V_{3}\cap V_{2}^{\bot}) {\big/} (V_{3}\cap V_{3}^{\bot}+ V_{2}\cap V_{2}^{\bot})= 0$; \\
$\Phi_{44}=(V_{4}\cap V_{3}^{\bot}) {\big/} (V_{4}\cap V_{4}^{\bot}+ V_{3}\cap V_{3}^{\bot})= \langle \partial_{1} \rangle$, ~$\Phi_{22}^{0}=0$, ~$n_{22}^{1}=1$; \\
As a result, $\omega=dx_{1}^{(2)}+ \ldots +dx_{n}^{(2)}$. Set of heights: $(3, \ldots, 3, 4)$.

\section{Filtered deformations of graded \\ non-alternating Hamiltonian Lie algebras}

The filtered Lie algebra $\mathscr{L}=\mathscr{L}_{-q}\supset \ldots \supset \mathscr{L}_{-1}\supset \mathscr{L}_{0} \supset \mathscr{L}_{1}\supset \ldots \mathscr{L}_s \supset 0$ is called the filtered deformation of the graded Lie algebra $L =\text{gr}\,\mathscr{L}$. ~Let ~$P^{(1)}(\mathscr{F}, \omega_{0}) \subseteq$  $ \subseteq L \subseteq \widetilde{P}(\mathscr{F}, \omega_{0})$ be the graded non-alternating  Hamiltonian Lie algebra, corresponding to the form $\omega$ with constant coefficients, $\mathscr{F}\colon E=E_{0}\supseteq E_{1}\supseteq \ldots \supseteq E_{r}\supset E_{r+1}=0$, $W=W(\mathscr{F})$, $V=L_{-1}= W_{-1}\cong W(\mathscr{F}){\big /}W(\mathscr{F})_{(0)}$, $E=V^{\ast}$.

According to Embedding Theorem ~\cite{K} there is a minimal embedding  $\tau\colon (\mathscr{L} \mathscr{L}_{0})\rightarrow (W(\mathscr{F}'), W(\mathscr{F}')_{(0)})$ with the minimal flag $\mathscr{F}'= \mathscr{F}(\mathscr{L}, \mathscr{L}_{0})$. The embedding $\tau$ is uniquely determined up to the automorphism of the pair $(W(\mathscr{F}'), W (\mathscr{F}')_{(0)})$. Such automorphisms are induced by admissible automorphisms of the algebra $O (\mathscr{F}')$. In particular, $\mathscr{F} (L, L_ {(0)})= \mathscr{F}$. It is known that $\mathscr{F}\leqslant \mathscr{F}'$ (~\cite{K}). The following theorem is a special case of the theorem proved in ~\cite{K}, ~\cite{K1}

\begin{theorem}\label{filt 1}
{\it Let $L= L_{-1}+ L_{0}+ \ldots $ be a transitive graded Lie algebra. If 
	\begin{itemize}
		\item[$I.~~$] $L_{-1}$ is an irreducible $L_{0}$ - module,
		\item[$ii.~$] $H^{1} (L_{0}, L_{-1})= 0$,
		\item[$iii.$] $\text{mtp} (L_{-1}, X(L_ {(1)})\leqslant m (\mathscr{F} (L, L_ {(0)})) - n$, $n=\dim L_{-1}$,
	\end{itemize}
	then for any filtered deformation of $\mathscr{L}$ of Lie algebra $L$ 
	\begin{itemize}
		\item [$(a)$] $\mathscr{F} (\mathscr{L}, \mathscr{L}_{0})= \mathscr{F} (L, L_ {(0)})$,
		\item [$(b)$] $\text{Der}\,\mathscr{L}\cong N_{\overline{W}(\mathscr{F})} (\tau(\mathscr{L})$.
	\end{itemize}
	Here $\tau\colon \mathscr{L}\rightarrow W (\mathscr{F})$ is the minimum embedding, $\overline{W} (\mathscr{F})$ is the $p$ -- closure of $W (\mathscr{F})$ in $\text{Der}\, O (\mathscr{F})$, $X(L_ {(1)})= L_ {(1)} {\big/}[L_ {(1)}, L_{(1)}]$, $\text{mtp}(Q, V)$ is the multiplicity of $L_{0}$-module $Q$ in $L_{0}$-module $V$.} 
	
	\raggedleft $\square$
\end{theorem}

Let $L$ be a non-alternating Hamiltonian Lie algebra, $P^{(1)} (\mathscr{F}, \omega_{0}) \subseteq L \subseteq \widetilde{P}(\mathscr{F}, \omega_{0})$,  $\omega_{0}$ be a non-degenerate non-alternating form with constant coefficients.

\begin{proposition}\label{filt 2}
{\it Let $\overline{L}_{0}$ be the $p$ - closure of $L_{0}$ in $W_{0}\cong gl(L_{-1})$, $V= W{\big /}W_{(0)}\cong L_{-1}$, $\overline{\omega}_{0}$ non-alternating symmetric bilinear form at $E=V^{\ast}$, dual $\omega_{0}$, $E^{0}$ subspace of isotropic vectors of $E$. The following statements are true.
	\begin{itemize}
		\item[$i.~$] $\overline{L}_{0}= \overline{L}_{0}^{(1)}\oplus T$, ~where $T$ -- torus, ~$T_{0} \subseteq T \subseteq T_{1}$, ~$T_{0}\cong \langle x_{i}^{(2)}+ x_{j}^{(2)}, ~i, j=1, \ldots, n \rangle$, ~$T_{1}\cong~\langle x_{i}^{(2)}, ~i=1, \ldots, n \rangle \subset \widetilde{P}_{0}$. Here $\{ x_{i} \}$ is the orthonormal basis of $E$ with respect to $\overline{\omega}_{0}$.
		\item[$ii.$] If $E_{1}\not\subset E^{0}$, then $T= T_{1}$,
		\item[$iii.$] If $n> 2$ or $n= 2$ and $T= T_{1}$, then $L_{-1}$ is an absolutely irreducible $\overline{L}_{0}$-module,
		\item[$iv.$] If $n> 4$ or $n= 2,3,4$ and $T= T_{1}$,
		\begin{itemize}
			\item[$(a)$] $H^{1} (\overline{L}_{0}, L_{-1})= 0$,
			\item [$(b)$] $(S^{3}(V)^{\ast})^{T}=0$.
			\item [$(c)$] $\text{mtp} (L_{-1}, X(L_ {(1)})\leqslant m(\mathscr{F} (L, L_ {(0)}))-n$.		\end{itemize}
	\end{itemize}}
\end{proposition}
\begin{proof} 
	$i.$ While studying the pair $(L_{-1}, L_{0})$ we can assume that the form $\omega_{0}$ has the form $\omega_{0}= (dx_{1})^{(2)}+ \ldots + (dx_{n})^{(2)}$. Then $\overline{\omega}_{0} (x_{i}, x_{j})= \delta_{ij}$ and $\{ f, g \}= \sum \partial_{i}f\partial_{i}g$. It is easy to check that $(\text{ad}(x_{i}x_{j}))^{2}= \text{ad}(x_{i}^{(2)}+ x_{j}^{(2)})$, which implies $i$.
	
	$ii.$ Let $y\in E_{1}$, $y\notin E^{0}$. Choose the basis $\{ y_{i} \}$ in $E$, coordinated with the flag $\mathscr{F}$, $y= y_{1}$. Let $m_{1}$ be the height of $y_{1}$. We have $m_{1}>1$. Therefore, $y_{1}^{(2)}\in L_{0}$. In the basis $\{ y_{i} \} $ $ \overline{\omega}_{0}$ looks like $\overline{\omega}_{0}(u, v)= \sum \limits_{i=1}^{n} \overline{\alpha}_{ii}\partial_{i}u\partial_{i}v + \sum\limits_{i, j} \overline{\alpha}_{ij}\partial_{i}u\partial_{j}v$, $u, v\in E$, $\overline{\alpha}_{11}= \overline{\omega}_{0}(y_{1}, y_{1})\neq 0$. Here $(\overline{\alpha}_{ij})= (\omega_{ij})^{-1}$.
	$$\{ f, g \}=\sum_{i=1}^{n} \overline {\alpha}_{ii}\partial_{i}f\partial_{i}g + \sum_{i, j}\overline{\alpha}_{ij} \partial_{i}f\partial_{j}g,$$
	$\overline{\alpha}_{ij}= \overline{\alpha}_{ji}$. Therefore, $\text{ad}\, y_{1}^{(2)}= \overline{\alpha}_{11}y_{1}\partial_{1} + \sum\limits_{j>1} \overline {\alpha}_{1j} y_{1}\partial_{j}$, $\text{tr}\left({\text{ad}\, y_{1}^{(2)}}\vert_{L_{-1}}\right)= \overline{\alpha}_{11}\neq 0$. Obviously, ${\text{tr}} \vert_{\overline{L}_{0}^{(1)}\oplus T_{0}} = 0$. Therefore $y_{1}^{(2)}= a+ t$, $a\in \overline{L}_{0}^{(1)}$, $t\in T$, $t\notin T_{0}$. Therefore, $T= T_{1}$. \\ [-0.2 cm]
	
	$iii.$ The statement is verified directly.
	
	$iv.$ If $T= T_{1}$, then $z= x_{1}^{(2)}+ \ldots + x_{n}^{(2)}\in Z(\overline{L}_{0})$ and $\text{ad} z|_{L_{-1}}= \text{id}$. So $H^{1} (\overline{L}_{0}, L_{-1})= 0$. Note that if $n$ is even, $z\in T_{0}\subset \overline{L}_{0}$.	
	Let $n> 3$ and $T= T_{0}= \langle h_{1}, \ldots, h_{n-1} \rangle$, $h_{i}= x_{i}^{(2)}+ x_{i+1}^{(2)}$, $\{ \varepsilon_{i} \}$ dual basis of $T_{0}^{\ast}$. Then the weights of $L_{-1}$ are $\alpha_{1}=\varepsilon_{1}$, $\alpha_{2}=\varepsilon_{1}+ \varepsilon_{2}$, $\ldots$ , $\alpha_{n-1}=\varepsilon_{n-2}+ \varepsilon_{n-1}$, $\alpha_{n}=\varepsilon_{n-1}$. All weights are nonzero and weight spaces are one-dimensional. The roots of $\overline{L}_{0}$ are the sums of different weights of $L_{-1}$. If $n> 4$, all root spaces are one-dimensional, the roots are different from the weights of $L_{-1}$. If $n= 4$, then the roots are different from the weights $L_{-1}$, the root spaces are two-dimensional. From the decomposition of the complex $C^{\ast}(\overline{L}_{0}, L_{-1}) $ into weight spaces with respect to $T$ we obtain
	$$H^{1} (\overline{L}_{0}, L_{-1})= H^{1}_{0}(\overline{L}_{0}, L_{-1})= H^{1}(C^{\ast}_{0}(\overline{L}_{0}, L_{-1})).$$
	Since all weights of $L_{-1}$ are nonzero and distinct from roots of $\overline{L}_{0}$, \\ $C^{\ast}_{0} (\overline{L}_{0}, L_{-1})=0$ and $H^{1} (\overline{L}_{0}, L_{-1})=0$. Note that for $n= 4$ $z\in T_{0}\subset \overline{L}_{0}$ and thus $H^{1}(\overline{L}_{0}, L_{-1})=0$.

Similarly, for ~$\varphi\in (S^{3}(V)^{\ast})^{T}$ ~we have ~$ \ varphi(v_{\mu_{1}}, v_ {\mu_{2}}, v_ {\mu_{3}})= 0$, ~if ~$\mu_{1}+\mu_{2}+ \mu_{3}= 0$, ~$\mu_{i}\in  \{ \alpha_{1}, \ldots, \alpha_{n} \}$.  The weights of module $V= L_{-1}$ satisfy the only relation $\alpha_{1}+ \ldots + \alpha_{n}= 0$. Thus, for $n> 3$ $\varphi= 0$. Let $T= T_{1}$. Then $T= \langle h_{1}, \ldots, h_{n} \rangle$. Weights of $L_{-1}$ are $\varepsilon_{1}, \ldots, \varepsilon_{n}$, where $\{ \varepsilon_{i} \}$ is the dual basis of $T^{\ast}$. Therefore, $\varphi= 0$. In particular, it is true for $n= 2$ and $3$. 
	The statement $(c)$ is proved by direct computations analogous to the computation of $X(L_ {(1)})$ for the classical Hamiltonian Lie algebra in ~\cite{KSh} (see ~\cite{KSh}, Chapter 3, Proposition 1). 
\end{proof}

\begin{theorem}\label{filt 3}
{\it Let $L= L_{-1}+ L_{0}+ \ldots $ be a graded non-alternating Hamiltonian Lie algebra, $P^{(1)}(\mathscr{F}, \omega_{0}) \subseteq L \subseteq \widetilde{P}(\mathscr{F}, \omega_{0})$, $\mathscr{L}$  filtered deformation of $L$. For $n= 2,3$, assume that $E_{1}\not\subset E^{0}$, where $E^{0}$ is the subspace of isotropic vectors with respect to the form $\overline{\omega}_{0}$ on the $E$ dual to the form $\omega_{0}$. For $n= 4$, assume that $\mathscr{F}$ is a nontrivial flag. Then
	\begin{itemize}
		\item[$i.$] $\mathscr{F} (\mathscr{L}, \mathscr{L}_{0})= \mathscr{F} (L, L_ {(0)})=\mathscr{F}$,
		\item[$ii.$] $\text{Der}\,\mathscr{L}\cong N_{\overline{W} (\mathscr{F})} (\tau (\mathscr{L}))$,
	\end{itemize}
	where $\tau\colon \mathscr{L}\rightarrow W (\mathscr{F})$ is the minimal embedding.}
\end{theorem}
\begin{proof}
	According to~\ref{filt 2}, the conditions of the theorem~\ref{filt 1} are satisfied for the Lie algebra $\overline{L}=L_{-1}+ \overline{L}_{0} +L_{1}+ \ldots $. Hence, for any filtered deformation $\overline {\mathscr{L}}$ of the Lie algebra $\overline{L}$
	\begin{equation}\label{eqeqeqeq}
	\mathscr{F} (\overline{\mathscr{L}}, \overline{\mathscr{L}}_{0})= \mathscr{F} (\overline{L}, \overline{L}_{(0)})=\mathscr{F}, ~\text{Der}\,\overline{\mathscr{L}}\cong N_{\overline{W} (\mathscr{F})}(\overline{\tau} (\overline{\mathscr{L}})),
	\end{equation}
	where $\overline{\tau}$ is the minimal embedding of $\overline{\mathscr{L}}$ in $W(\mathscr{F})$.
	
	Let now $\mathscr{L}$ be a filtered deformation of the Lie algebra $L$, $\tau\colon \mathscr{L}\rightarrow W (\mathscr{F}')$ minimal embedding, $\mathscr{F}'= \mathscr{F} (\mathscr{L}, \mathscr{L}_{0})$. Identify $\mathscr{L}$ and $\tau(\mathscr{L})$. Let $\overline{L}_{0}= L_{0}+ \langle a_{i}^{2}, ~i=1, \ldots, s \rangle$, ~$a_{i}\in L_{0}$, ~$l_{i}\in \mathscr{L}_{0}$, ~$l_{i} \equiv a_{i} (\text{mod}\, W_ {(1)})$. ~As to ~$W(\mathscr{F}')_{(0)}$ is a $p$-subalgebra of $\text{Der}\,O(\mathscr{F}')$, ~$l_{i}^{2}\equiv  a_{i}^{2} (\text{mod}\,W_{(1)})$. Hence, $\overline{\mathscr{L}}= \mathscr{L}+ \langle l_{i}^{2}, ~i=1, \ldots, s \rangle$ is a filtered deformation of the Lie algebra $\overline{L}$. Thus, $\mathscr{F}'\geqslant \mathscr{F} (\overline{\mathscr{L}}, \overline{\mathscr{L}_{0}})$ (see ~\cite{K}). On the other hand, $\overline{\tau}_{\mathscr{L}}\colon \mathscr{L}\rightarrow W(\mathscr{F})$ is embedding, where $\overline{\tau}\colon \overline{\mathscr{L}}\rightarrow W(\mathscr{F})$ is the minimal embedding of $\mathscr{L}$, and $\mathscr{F}'\leqslant \mathscr{F}$. Thus $\mathscr{F}'= \mathscr{F}$ and we can assume that $\tau= \overline{\tau}$. Identifying $\mathscr{L}$ and $\text{ad}\,\mathscr{L}\subset \text{Der}\,\mathscr{L}$, we get $[D, \text{ad}(l^{2})]= \text{ad}[D(l), l]$. Hence, $D\in \text{Der}\,\mathscr{L}$ may be extended up to derivation of the Lie algebra $\overline{\mathscr{L}}$  in such a way that $D(l_{i}^{2})= [D(l_{i}), l_{i}]$. From here we obtain, using \eqref{eqeqeqeq}, that $\text{Der}\,\mathscr{L}\cong N_{\overline{W}(\mathscr{F})}(\tau(\mathscr{L})$.
\end{proof}

The proof of the following theorem is similar to the proof of Theorem~8.1~\cite{Sk1} (or Theorem~7.1~\cite{Sk2}) for the classical Hamiltonian Lie algebras based on the theory of truncated coinduced modules ~\cite{Blattner}, ~\cite{K}, ~\cite{St} and the theory of modular pairs of Lie-Cartan~ (see \cite{Sk1} and ~\cite{Sk2} for detailed presentation).

\begin{theorem}\label{filt 4}
{\it Under the conditions of theorem~\ref{filt 3}, there exists a unique up to an admissible automorphism of the algebra $W(\mathscr{F})$ a non-alternating Hamiltonian form $\omega$ with coefficients from $O(\mathscr{F})$, such that $\omega(0) =\omega_{0}$ and $P^{(1)}(\mathscr{F}, \omega)\subseteq \mathscr{L} \subseteq \widetilde{P} (\mathscr{F}, \omega)$.}
\end{theorem}
\begin{proof}
 We note only the changes in the proof of Theorem~8.1~\cite{Sk1} that need to be done to extend Theorem 8.1~\cite{Sk1} to the case of filtered deformations of graded non-alternating Hamiltonian Lie algebras.  First, we will consider the Lie algebra $\overline{\mathscr{L}}$ as in the proof of the theorem~\ref{filt 3}, which is a filtered deformation of the graded Lie algebra $\overline{L}$ (see~\ref{filt 2}), $\mathscr{L} \ subset \overline {\mathscr{L}}$. Second, in the proof of theorem~8.1~\cite{Sk1} we need to replace $\overline{P}=\Lambda^{2}V{\big /}\mathscr{L}_{0} \cdot \Lambda^{2}V$ with $\overline{P}=S^{2}V {\big /} \overline{\mathscr{L}} _ {0}\cdot S^{2}V$. Third, instead of $\varphi\colon \Lambda^{3}V\rightarrow \overline{P}$, consider $\varphi\colon S^{3}V\rightarrow \overline{P}$. All the information necessary to apply the proof of Theorem ~8.1~\cite{Sk1} in the case of filtered deformation of a non-alternating Hamiltonian Lie algebra is given in~\ref{filt 2}.
	
     As a result, we obtain that
	$$P^{(1)}(\mathscr{F}, \omega)\subseteq \overline{\mathscr{L}} \subseteq \widetilde{P}(\mathscr{F}, \omega),$$
	$\omega(0) =\omega_{0}$, $\omega$ is a non-alternating Hamiltonian form on $W(\mathscr{F})$ with values in some invertible module of the de Rham coefficients of the pair $(O(\mathscr{F}), W(\mathscr{F})$. It is easy to check that any invertible module of the coefficients of the de Rham $P$ of the pair $(O(\mathscr{F}), W(\mathscr{F}))$ is isomorphic to the submodule of $O(\mathscr{F})u\subset \widehat{O}(E)$, $u=\exp f$, $f\in \langle x_{i}^{(p^{m_{i}})}, ~i=1, \ldots, n \rangle$, where $\{ x_{i} \} $ is a basis of $E$, coordinated with the flag $\mathscr{F}$, $m_{i}$ is the height of $x_{i}$. However, Lemma~\ref{lemma 0.2} implies that $f= 0$. Hence, $P= O (\mathscr{F})$.
\end{proof}

\section{Non-alternating Hamiltonian forms with \\ non-constant coefficients}

Let $\mathfrak{m} (\mathscr{F})^{(j)}$ be the standard filtration of $O (\mathscr{F})$. Thus, $\mathfrak{m}(\mathscr{F})^{(1)}=\mathfrak{m}(\mathscr{F})$.  Denote by $G (\mathscr{F})$ the group of admissible automorphisms of the algebra $O (\mathscr{F})$. Let $G'(\mathscr{F})$ be a subgroup of automorphisms $\sigma\in G (\mathscr{F})$ such that $\sigma f - f \in \mathfrak{m} (\mathscr{F})^{2}$ for all $f \in O (\mathscr{F})$. The group $G (\mathscr{F})$ preserves ideals $(\mathscr{F})^{(k)}$. Define the filtration of corresponding groups by normal subgroups,  $j\geqslant 0$
$$  G(\mathscr{F})_{j}= \{ \sigma\in G(\mathscr{F}) ~ | ~ \sigma f-f \in \mathfrak{m} (\mathscr{F})^{(j+l)} ~~ \forall f \in \mathfrak{m} (\mathscr{F})^{(l)}, ~l\geqslant 0 \}, ~j\geqslant 0$$
$$G'(\mathscr{F})_{j}=G (\mathscr{F})_{j}\cap G'(\mathscr{F}),$$
and the following Lie subalgebras of $W (\mathscr{F})$
$$\mathfrak{g}'(\mathscr{F})=\mathfrak{m}(\mathscr{F})^{2}W(\mathscr{F}), \quad \mathfrak{g}'(\mathscr{F})_{j}=(\mathfrak{m}(\mathscr{F})^{2} \cap \mathfrak{m}(\mathscr{F})^{(j+1)}) W (\mathscr{F}).$$

Put 
$$z_{i}=x_{i}^{(2^{m_{i}})}, \quad \langle x_{1}, \ldots, x_{n} \rangle=E, \quad m_{i}=\min\{ j ~|~ x_{i}\notin E_{j} \}. $$

Note that the group $G'(\mathscr{F})$ acts trivially in $H^{\ast} (\Omega (\mathscr{F})$ (this follows from the triviality of the action on $H^1 (\Omega (\mathscr{F}))\equiv tilde{O} (\mathscr{F})/O (\mathscr{F})$)).

\begin{proposition}[~\cite{S}]\label{prop 1.24}
{\it For $j\geqslant 1$ for a given $\sigma \in G'(\mathscr{F})_{j}$ there exists the unique  $D \in \mathfrak{g}'(\mathscr{F})_{j}$  such that 
	\begin{gather}\label{for 1.36}
	\begin{array}{c}
	\sigma x= x + Dx \qquad \forall x\in E, \\
	(\sigma - id-D)(\mathfrak{m}(\mathscr{F})^{(l)})\subseteq \mathfrak{m}(\mathscr{F})^{(j+l+1)} \qquad l\geqslant 0.
	\end{array}	
	\end{gather} \\ [-1.3 cm]}
	\begin{flushright}$\square$\end{flushright} 	
\end{proposition}

\begin{collor}[see ~\cite{S}]\label{lem 1.25}
{\it Let $\psi \in \mathfrak{m} (\mathscr{F})^{(r)} S\Omega^{k}(\mathscr{F})$, ~$\sigma \in G'(\mathscr{F})_{j}$, where $j\geqslant 1$. ~Then ~$\sigma\psi - \psi \in \mathfrak{m}(\mathscr{F})^{(j+r)}S\Omega^{k}(\mathscr{F})$. Moreover, if $D \in \mathfrak{g}'(\mathscr{F})_{j}$ is related to $\sigma$ by \eqref{for 1.36}, $\sigma\psi - \psi - D\psi \in \mathfrak{m} (\mathscr{F})^{(j+r+1)}S\Omega^{k} (\mathscr{F})$.}
\end{collor}

For the non-alternating Hamiltonian form $\omega \in S\Omega^{2}(\mathscr{F})$, we define the isomorphism of $O(\mathscr{F})$-modules $i_{\omega}\colon W(\mathscr{F}) \rightarrow S\Omega^{1}(\mathscr{F}),$ $i_{\omega}(D) = D\lrcorner\omega$.

In the future, we use abbreviations ~$G'_{j}= G'(\mathscr{F})_{j}$, ~$S\Omega= S\Omega (\mathscr{F})$, ~$\mathfrak{m}= \ = \ mathfrak{m} (\mathscr{F})$, etc. 

Write non-alternating Hamiltonian form $\omega$ in the form $\omega= \sum a_{ii}dx_{i}^{(2)} + \sum\limits_{i<j} a_{ij}dx_{i}dx_{j} + \\ + d\varphi + \sum\limits_{i<j} b_{ij}dz_{i}dz_{j}$, where $a_{ij},b_{ij}\in K$, $\varphi \in \mathfrak{m}^{(2)}S\Omega^{1}$.
We assume that $\omega (0)$ is given to \textit{canonical form}.

We identify $H^{2}(\Omega)$ with the subspace $\langle dz_{i}dz_{j} \rangle \subset H^{2}(S\Omega)$ and define the semi-linear map $\lambda\colon H^{2}(\Omega)\rightarrow O(\mathscr{F})$ assuming $\lambda(b^{2}dz_{i}dz_{j})= \\ = bx_{i}^{(2^{m_{i}-1})}x_{j} {(2^{m_{j}-1})}$.

\begin{lemma}\label{lem dx2}
{\it (1) An automorphism $\sigma \in G'$ acts identically on the cohomological class of the form $\omega$ if and only if for all $i$, such that $\omega(0)$ contains $dx_{i}^{(2)}$, runs in $\sigma x_{i} = x_{i} + f$ where $f \in \mathfrak{m}^{2}$ and $f$ does not contain monomials of the form $\lambda(dz_{s}dz_{j})$. 
	
	(2) Let $\{ dx_{i}^{(2)}, dz_{i}dz_{j}\} $ be the basis of $H^{2}(S\Omega)$. The automorphism $\sigma \in G'$ can act not identically only on the elements $dx_{i}^{(2)}$. \\ If $\sigma x_{i} = x_{i} + bx_{s}^{(2^{m_{s}-1})}x_{r}^{(2^{m_{r}-1})}$, then $\sigma dx_{i}^{(2)} = dx_{i}^{(2)} +  d\varphi + b^{2}dz_{s}dz_{r}$, where $\varphi= bx_{s}^{(2^{m_{s}-1})} x_{r}^{(2^{m_{r}-1})} dx_{i}$.}
\end{lemma}
\begin{proof}
	According to theorem~\ref{ter 0.1} $H^{2} (S\Omega)= H^{2}(\Omega) \oplus \langle dx_{1}^{(2)}, \ldots, dx_{n}^{(2)} \rangle$. On $H^{2}(\Omega)$, the automorphism $\sigma \in G'$ acts trivially. 
	
	Let $\sigma \in G_{j}'$, $j\geqslant 1$, $\sigma x = x+ Dx$ for $x\in E$, $D \in \mathfrak{g}_{j}'$. Let $Dx=\sum a_{\alpha}x^{(\alpha)}$. Then $\sigma (dx)^{(2)}= (\sigma dx)^{(2)}= (d\sigma x)^{(2)}= (d(x+ Dx))^{(2)}= (dx+ dDx)^{(2)}= (dx)^{(2)}+ dx\cdot dDx+ (dDx)^{(2)}= (dx)^{(2)}+ d(D(x)dx) + (d\sum a_ {\alpha}x^{(\alpha)})^{(2)} = (DX)^{(2)}+ d(D (x)dx) + \left (\sum\limits_{\alpha,i} a_{\alpha}\partial_{i}x^{(\alpha)}dx_{i}\right) ^{(2)} = (dx)^{(2)}+ \\ + d (D(x) dx) + d\varphi' + \sum a_{sr}dz_{s}dz_{r}= (dx)^{(2)} + d\varphi + \eta$, where $a_{sr}= a_{\alpha}^{2}$ is coefficient square at $x^{(\alpha)}= x_{s}^{(2^{m_{s}-1})}x_{r}^{(2^{m_{r}-1})}$ in $Dx$. Consequently, the cohomology class of the form $(dx)^{(2)}$ changes only if $Dx$ has a summand $a_{\alpha} x_{s}^{(2^{m_{s}-1})}x_{r}^{(2^{m_{r}-1})}= \lambda (a_{sr}dz_{s}dz_{r})$. For $Dx = b\lambda (dz_{s}dz_{r})$ we get $\sigma(dx)^{(2)} = (dx)^{(2)} + d(D (x)dx) + b^{2}dz_{s}dz_{r}$.
\end{proof}

\begin{collor}\label{col dx2}
{\it For $\sigma \in G'$ such that $\sigma x_{i} = x_{i} + bx_{s}^{(2^{r})}x_{j}^{(2^{t})}$, we have $\sigma dx_{i}^{(2)} = dx_{i}^{(2)} + d\varphi + \ + b^{2}x_{s}^{(2^{r+1}-1)} x_{j}^{(2^{t+1}-1)} dx_{s}dx_{j}$, where $\varphi= bx_{s}^{(2^{r})} x_{j}^{(2^{t})} dx_{i}$. Here $0\leqslant r < m_{s}, ~0\leqslant t < m_{j}$.}
\end{collor}

\begin{lemma}\label{lem 2.3}
{\it Let $\omega \in Z^{k}(S\Omega)$, $k=1,2$ and $\sigma \in G_{j}'$, $D \in \mathfrak{g}_{j}'$, where $j\geqslant 1$, are bound by the condition \eqref{for 1.36}. If $\omega \in Z^{1}(S\Omega)$, then there is $\varphi \in \mathfrak{m}$ such that $\sigma\omega - \omega = d\varphi$ and $\varphi - D\lrcorner\omega \in \mathfrak{m}^{(j+2)}$. If $\omega \in Z^{2}(S\Omega)$, then $\sigma\omega - \omega = d\varphi+ \eta$, where $\varphi \in S\Omega^{1}$, $\eta \in H^{2} (\Omega)$, $ \lambda(\eta)\in \mathfrak{m}^{(j+1)}$, and $\varphi-D\lrcorner\omega \in \mathfrak{m}^{(j+2)}s\Omega^{1}$.}
\end{lemma}
\begin{proof}
	Consider first the action of $\sigma$ on $\omega = d\psi$, where $\psi \in S\Omega^{k-1}$. Correcting $\psi$ on coboundary, we can assume that $\psi \in \mathfrak{m}S\Omega^{k-1}$. Assuming $\varphi= \sigma\psi - \psi - d(D\lrcorner\psi)$, we get $\sigma d\psi - d\psi = d(\sigma\psi - \psi) = d\varphi$. Since $D\lrcorner d\psi = D\psi -d(D\lrcorner\psi)$, by corollary~\ref{lem 1.25}, $\varphi - D\lrcorner d\psi= \sigma\psi - \psi - D\psi \in \mathfrak{m}^{(j+2)}S\Omega^{k-1}$.
	
	Now let $\omega=(dx)^{(2)}$, $x\in E$. We use notations from the proof of the Lemma~\ref{lem dx2}. For $Dx=\sum a_{\alpha}x^{(\alpha)} \in \mathfrak{m}^{2} \cap \mathfrak{m}^{(j+1)}$ we have 	
	$\sigma(dx)^{(2)}= (dx)^{(2)} + d\varphi + \eta$, where $\lambda (\eta)= \lambda (\sum a_{is}dz_{i}dz_{s})\in \mathfrak{m}^{(j+1)},$ and $\varphi - D\lrcorner (dx)^{(2)}= \varphi-D(x)dx= \varphi'\in \mathfrak{m}^{(2j+1)}S\Omega^{1} \subseteq \mathfrak{m}^{(j+2)}S\Omega^{1}$.
	
	Show that for $\omega \in Z^{1}(S\Omega)$ the Lemma is true. For $f \in \mathfrak{m}$ 
	$$\sigma d(f^{(2)})- d(f^{(2)})= d ((\sigma f)^{(2)} - f^{(2)})= dh,$$
	where $h= (\sigma f-f)^{(2)} + f (\sigma f-f) \in O (\mathscr{F})$, because $\sigma f-f\in \mathfrak{m}^{2}$. Herewith
	\begin{equation}\label{for 2.3}
	h - D \lrcorner d(f^{(2)}) = h - Df^{(2)} = h-fDf = f(\sigma f - f-Df) + (\sigma f-f)^{(2)}. 
	\end{equation} 
	In view of \eqref{for 1.36} $f (\sigma f-f - Df)\in \mathfrak{m}^{(j+3)}$ and $(\sigma f-f)^{(2)}\in \mathfrak{m}^{(2j+2)}$, so that the right side of \eqref{for 2.3} is known to be $\mathfrak{m}^{(j+2)}$, that is, for $d(f^{(2)})$ the Lemma is true. But any 1-cocycle comparable  modulo coboundaries with the appropriate cocycle $fdf$. So the Lemma is true for all $\omega \in Z^{1}(S\Omega)$.
	
	It remains to prove that the Lemma is true for $\omega_{1}\omega_{2} \in Z^{1}(S\Omega)\cdot Z^{1}(S\Omega)$. Let $\sigma\omega_{i} - \omega_{i} = d\varphi_{i}$. Then $\sigma (\omega_{1}\omega_{2}) - \omega_{1}\omega_{2} = (\sigma \omega_{1} - \omega_{1})(\sigma \omega_{2} - \omega_{2}) + \omega_{1}(\sigma \omega_{2} - \omega_{2}) + (\sigma \omega_{1} - \omega_{1})\omega_{2} = d\varphi_{1}d\varphi_{2} + \omega_{1}d\varphi_{2} +  d\varphi_{1}\omega_{2} = d(\varphi_{1}d\varphi_{2}) + d(\omega_{1}\varphi_{2} + \varphi_{1}\omega_{2})= d\varphi$ and 
	$\varphi_{1}d\varphi_{2} + \omega_{1}\varphi_{2} + \varphi_{1}\omega_{2} - D\lrcorner(\omega_{1}\omega_{2}) = \varphi_{1}d\varphi_{2} + \omega_{1}\varphi_{2} + \\ + \varphi_{1}\omega_{2} (D\lrcorner\omega_{1})\omega_{2} - \omega_{1}(D\lrcorner\omega_{2}) = \varphi_{1}d\varphi_{2} + \omega_{1}(\varphi_{2} - D\lrcorner\omega_{2}) + (\varphi_{1} - D\lrcorner\omega_{1})\omega_{2}$. As in corollary~\ref{lem 1.25} ~$d\varphi_{i}\in \mathfrak{m}^{(j)}S\Omega^{1}$, then $\varphi - D\lrcorner(\omega_{1}\omega_{2}) \in \mathfrak{m}^{(2j+1)}S\Omega^{1} + \mathfrak{m}^{(j+2)}S\Omega^{1} + \mathfrak{m}^{(j+2)}S\Omega^{1} \subseteq \mathfrak{m}^{(j+2)}S\Omega^{1}$.
	
	It follows from the Theorem~\ref{ter 0.1} that $Z^{2} (S\Omega)$ is spanned by $(dx_{i})^{(2)}$ ($i=1,\ldots, n$), $Z^{1}(S\Omega)\cdot Z^{1}(S\Omega)$, and $B^{2}(S\Omega)$. So, the Lemma is proven.
\end{proof}

\begin{collor}\label{col 2.4}
{\it If $\omega \in S\Omega^{2}$ is a closed form, $\sigma \in G_{j}'$, $j\geqslant 1$, then $\sigma\omega - \omega = d\varphi + \eta$ for a suitable form $\varphi \in \mathfrak{m}^{(j+1)} S\Omega^{1}$ and $\eta \in H^{2}(\Omega)$, $\lambda(\eta)\in \mathfrak{m}^{(j+1)}$.}
\end{collor}

For the non-alternating Hamiltonian form $\omega = \omega(0) + d\varphi + \eta$, where $\varphi \in \mathfrak{m}^{(2)} S\Omega^{1}$ and $\eta \in \ H^{2} (\Omega)$ we introduce the set of indices $I= 
\{ i ~ | ~ \bar{a}_{ii} \neq 0\}$. We also need the set $$\widetilde{\mathfrak{m}}^{(j)} S\Omega^{1}= \langle T \rangle, \text{ where} $$
$$T = \{ x^{(\alpha)} dx_{k} ~ | ~  |\alpha|\geqslant j, ~k=1,\ldots, n \} ~\diagdown~ \{ x_{r}^{(2^{l})} x_{s}^{(2^{t})} dx_{i}, ~x_{q}^{(2)} dx_{i}, ~i,q\in I \}.$$

\begin{lemma}\label{lem mm}
{\it Let $j\geqslant 1$ and $\varphi= x_{r}^{(2^{l})}x_{s}^{(2^{t})}dx_{i} \in \mathfrak{m}^{(j+1)}S\Omega^{1}$, $i\in I$. If $r,s\notin I$, or $l, t> 0$, or $r\notin I$, $t> 0$, or $s\notin I$, $l> 0$, then $\varphi = d\psi + \widetilde{\varphi}$, where 
$\widetilde{\varphi} \in \widetilde{\mathfrak{m}}^{(j+1)}S\Omega^{1}$. For $i= r$ and $m_{i}>1$ if $l> 0$ or $t> 1$, or $s\notin I$, then $x_{i}^{(2^{l})}x_{s}^{(2^{t})} dx_{i} = d\psi + 
\widetilde{\varphi}$, where $\widetilde{\varphi} \in \widetilde{\mathfrak{m}}^{(j+1)}S\Omega^{1}$.}
\end{lemma}
\begin{proof}
	If $i\neq r, s$, then $$d (x_{r}^{(2^{l})}x_{s}^{(2^{t})} dx_{i})= x_{r}^{(2^{l}-1)} x_{s}^{(2^{t})} dx_{r}dx_{i} + x_{r}^{(2^{l}}) x_{s}^{(2^{t}-1)} dx_{s}dx_{i} =$$ $$ =d(x_{r}^{(2^{l}-1)}x_{s} {(2^{t})} x_{i}dx_{r} + x_{r} {(2^{l})} x_{s}^{(2^{t}-1)} x_{i}dx_{s})$$. Therefore, $$x_{r}^{(2^{l})}x_{s}^{(2^{t})} dx_{i} = d\psi + x_{r}^{(2^{l}-1)} x_{s}^{(2^{t})}x_{i}dx_{r} + x_{r}^{(2^{l})} x_{s}^{(2^{t}-1)}x_{i}dx_{s}.$$ Obviously, $x_{r}^{(2^{l}-1)}x_{s}^{(2^{t})}x_{i}dx_{r} + x_{r}^{(2^{l})} x_{s}^{(2^{t}-1)} x_{i}dx_{s} \in \widetilde{\mathfrak{m}}^{(j+1)}S\Omega^{1}$, except for $r\in I$, $l=0$ or $s\in I$, $t= 0$. 
	
	If $m_{i}>1$, then $d(x_{i}^{(2^{l})} x_{s}^{(2^{t})} dx_{i}) = x_{i}^{(2^{l})} x_{s}^{(2^{t}-1)} dx_{s}dx_{i} = d(x_{i}^{(2^{l}+1)} x_{s}^{(2^{t}-1)}dx_{s})$. Therefore, $x_{i}^{(2^{l})}x_{s}^{(2^{t})}dx_{i} = d\psi + x_{i}^{(2^{l}+1)}x_{s}^{(2^{t}-1)}dx_{s}$. Obviously, $$x_{i}^{(2^{l}+1)} x_{s}^{(2^{t}-1)} dx_{s} \in \widetilde{\mathfrak{m}}^{(j+1)} S\Omega^{1}$$ 
for $l> 0$ or $t> 1$, or $s\notin I$.
\end{proof}

\begin{lemma}\label{lemma h}
{\it If $h\in \mathfrak{m}^{2}$, then $(dh)^{(2)}= d\varphi$, where $\varphi \in \widetilde{\mathfrak{m}}^{(2)}S\Omega^{1}$, then and only then, when $h$ does not contain monomials $\lambda(dz_{r}dz_{s})$ and monomials of the form $x_{q}x_{j}^{(2^{t})}$, where $q\in I$, $m_{q}= 1$ or $q, j\in I$, $m_{q}> 1$, $t= 0$.}
\end{lemma}
\begin{proof}
	Since $h\in \mathfrak{m}^2$, $h= \sum \beta_{j}g_{j}$ is a linear combination of decomposable monomials $g_{j}$. Then 
$$(dh)^{(2)} = \left( \sum \beta_{j}dg_{j}\right)^{(2)} = \sum\limits_{j<k} \beta_{j}\beta_{k}dg_{j}dg_{k} + \sum\limits_{j} \beta_{j}^{2}(dg_{j})^{(2)} = $$
$$\sum\limits_{j<k} d(\beta_{j}\beta_{k}g_{j}dg_{k}) + \sum\limits_{j} \beta_{j}^ {2} (dg_{j})^{(2)} = d\psi + \sum \beta_{j}^{2}(dg_{j})^{(2)},$$
where $\psi \in \widetilde{\mathfrak{m}}^{(2)}S\Omega^{1}$ because the form $\xi \notin 
\widetilde{\mathfrak{m}}^{(2)} S\Omega^{1}$ must contain $x_{r}^{(2^{l})} x_{s}^{(2^{t})} dx_{i}$ which is the product of two indecomposable elements and $dx$, or $x_{q}^{(2)}dx_{i}$ which is a product of an indecomposable element and $dx$. If a monomial $g$ is $y_{1}y_{2}$ and one of the monomials $y_{1} \in \mathfrak{m}^{2}$ or $y_{2} \in \mathfrak{m}^{2}$ decompose, then $(dg)^{(2)}= y_{1}d(y_{2}) + d (y_{1}) y_{2} = d(y_{1}^{(2)})d(y_{2}^{(2)}) = 0$.
	Let $g = x_{r}^{(2^{k})}x_{s}^{(2^{t})}$, then $(dg)^{(2)} d((x_{r}^{(2^{k})})^{(2)})d((x_{s}^{(2^{t})})^{(2)})= x_{r}^{(2^{k+1}-1)}x_{s}^{(2^{t+1}-1)}dx_{r}dx_{s} = d\varphi_{1}$ if $t< m_{s}-1$, or $(dg)^{(2)}= d\varphi_{2}$ if $k< m_{r}-1$. Here $\varphi_{1} = x_{r}^{(2^{k+1}-1)} x_{s}^{(2^{t+1})} dx_{r}$, $\varphi_{2} = x_{r}^{(2^k+1)} x_{s}^{(2^{t+1}-1)} dx_{s}$. By Lemma ~\ref{lem mm} $\varphi_{1} + df_{1} \in \widetilde{\mathfrak{m}}^{(2)}S\Omega^{1}$ if $r\notin I$ or $k> 0$, or $m_{r}> 1$, $s\notin I$, or $m_{r}> 1$, $t> 0$ and $\varphi_{2} + df_{2} \in \widetilde{\mathfrak{m}}^{(2)}S\Omega^{1}$ if $s\notin I$ or $t> 0$, or $m_{s}> 1$, $r\notin I$, or $m_{s}> 1$, $k> 0$.
	Thus, if $h$ does not contain monomials of the form $x_{r}^{(2^{m_{r}-1})} x_{s}^{(2^{m_{s}-1})} = \lambda(dz_{r}dz_{s})$ and monomials of the form $x_{q}x_{j}^{(2^{t})}$, where $q\in I$, $m_{q}= 1$ or $m_{q}> 1$, $t= 0$, $q, j\in I$, then $\varphi \in \widetilde{\mathfrak{m}}^{(2)}s\Omega^{1}$. The sufficiency is proved.
	
	Now, let $\sigma \in G', ~\sigma x_{i}= x_{i} + h$, then $\sigma dx_{i}^{(2)}= dx_{i}^{(2)} + dhdx_{i} + (dh)^{(2)}$ (see proof of Lemma~\ref{lem dx2}). If $h$ contains $\lambda(dz_{r}dz_{s})$, then the Lemma~\ref{lem dx2} $(dh)^{(2)}$ contains $dz_{r}dz_{s} \neq d\varphi$. If $h$ contains $x_{q}x_{j}^{(2^{t})}$, where $q\in I$, $m_{q}= 1$ or $q, j\in I$, $m_{q}> 1$, $t= 0$, then from corollary ~\ref{col dx2} and the reasoning above, we obtain that $(dh)^{(2)}$ contains $x_{q}x_{j}^{(2^{t+1}-1)}dx_{q}dx_{j} = d\psi$, where $\psi = x_{q}x_{j}^{(2^{t+1})} dx_{q}$ or $\psi = x_{q}^{(2)} x_{j}dx_{j}$, and $\psi \notin \widetilde{\mathfrak{m}}^{(2)}S\Omega^{1}$. That is, the necessity is proved.	
\end{proof}

\begin{proposition}\label{lem 2.5.1}
{\it Let $\omega, \omega'$ be two non-alternating Hamiltonian forms, and $\omega' -  \omega = d\varphi$, where $\varphi \in \widetilde{\mathfrak{m}}^{(j+1)}S\Omega^{1}$, $j\geqslant 1$. Then, there exists $\sigma \in G_{j}'$ such that $\omega' - \sigma\omega = d\widetilde{\varphi}$ for some $\widetilde{\varphi} \in \widetilde{\mathfrak{m}}^{(j+2)}S\Omega^{1}$.}
\end{proposition}
\begin{proof}
	Let $\varphi= \varphi_{1} + \varphi_{2}$, where $\varphi_{1} \in (\mathfrak{m}^{2}\cap \widetilde{\mathfrak{m}}^{(j+1)})S\Omega^{1}$, $\varphi_{2} \in \widetilde{\mathfrak{m}}^{(j+1)}S\Omega^{1}$, $\varphi_{2} \notin \mathfrak{m}^{2}S\Omega^{1}$. \\ Then $\varphi_{2}$ consists of $\beta x_{k}^{(2^{t})}dx_{s}$ and $t>1$terms.
	Correcting $\varphi_{2}$ on coboundary, $\varphi_{2}= d(\sum \beta x_{k}^{(2^{t})}x_{s}) - \sum x_{s}d(\beta x_{k}^{(2^{t})})= dg+ \widetilde{\varphi}_{2}$ and $\widetilde{\varphi}_{2} \in (\mathfrak{m}^{2}\cap \widetilde{\mathfrak{m}}^{(j+1)})S\Omega^{1},$ we can assume that $\varphi \in (\mathfrak{m}^{2}\cap \widetilde{\mathfrak{m}}^{(j+1)})S\Omega^{1}$. Since $i_{\omega}$ is an isomorphism of $O$ - modules, there is $D \in (\mathfrak{m}^{2}\cap \mathfrak{m}^{(j+1)})W = \mathfrak{g}_{j}'$ such that $D\lrcorner \omega=\varphi$. According to~\ref{prop 1.24} there is $\sigma \in G_{j}'$ such that \eqref{for 1.36} holds.
	
	Denote $Dx_{r}=h_{r}\in \mathfrak{m}^{2}\cap \mathfrak{m}^{(j+1)}$. The form $\omega$ is divided into $\omega(0)$ and summands of the form $dz_{k}dz_{r}$ and $dfdx_{k}$, $f \in \mathfrak{m}^{(2)}$. Since 
	$$D\lrcorner dz_{k}dz_{r} = x_{k}^{(2^{m_{k}}-1)} x_{r}^{(2^{m_{r} -1)}} h_{k} dx_{r} + x_{k}^{(2^{m_{k}} -1)} x_{r} {(2^{m_{r}} -1)} h_{r}dx_{k}$$
where each term can be decomposed into the product of the four elements and $dx$ and $D\lrcorner dfdx_{k} ~= ~h_{k}df ~+ ~D(f)dx_{k}$, where each summand can be decomposed into the product of three elements and $dx$, then $D\lrcorner (\omega - \omega(0))\in \widetilde{\mathfrak{m}}^{(j+1)}S\Omega^{1}$ as the form $\xi \notin \widetilde{\mathfrak{m}}^{(j+1)}S\Omega^{1}$ must contain $x_{r}^{(2^{l})}x_{s}^{(2^{t})}dx_{i}$ which is the product of two irreducible elements and $dx$, or $x_{q}^{(2)}dx_{i}$ which is the product of an irreducible element and $dx$. Thus, $D\lrcorner \omega\in \widetilde{\mathfrak{m}}^{(j+1)}S\Omega^{1}$ should be $D\lrcorner \omega(0)\in \widetilde{\mathfrak{m}}^{(j+1)}S\Omega^{1}$. Then if $dx_{i}^{(2)}$ is included in $\omega(0)$, then $h_{i}$ does not contain monomials of the form $x_{r}^{(2^{m_{r}-1})} x_{s}^{(2^{m_{s}-1})}= \lambda(dz_{r}dz_{s})$.
	By corollary~\ref{col 2.4} and Lemma~\ref{lem dx2} $\sigma\omega - \omega = d\psi$ for some $\psi \in \mathfrak{m}^{(j+1)}s\Omega^{1}$. 
	
	Now prove that $\psi \in \widetilde{\mathfrak{m}}^{(j+1)}S\Omega^{1}$. To check $dz_{k}dz_{r}$, we use the formula 
$$\sigma d(x_{k}^{(2^{m_{k}})}) - d(x_{k}^{(2^{m_{k}})}) = d ((\sigma x_{k})^{(2^{m_{k}})} - 
x_{k}^{(2^{m_{k}})}) = d ((x_{k} +h_{k})^{(2^{m_k})} - x_{k}^{(2^{m_k})}) = $$ 
$$ d\sum\limits_{l=1}^{2^{m_{k}}} x_{k}^{(2^{m_{k}}-l)}h_{k}^{(l)}= d\varphi_{k}$$
where $\varphi_{k} \in O(\mathscr{F})$, since $h_{k}\in \mathfrak{m}^{2}$. We find that $\sigma (dz_{k}dz_{r})- dz_{k}dz_{r} = d(\varphi_{k}d\varphi_{r} + \varphi_{k}dz_{r} + \varphi_{r}dz_{k})$ and $\varphi_{k}d\varphi_{r} + \varphi_{k}dz_{r} + \varphi_{r}dz_{k} = \psi_{1}$ and each term can be decomposed into the product of three elements, and $dx$, i.e. $\psi_{1}\in \widetilde{\mathfrak{m}}^{(j+1)}S\Omega^{1}.$ 
	
	Consider $dfdx_{k}$, $f \in \mathfrak{m}^{(2)}$. Let $\psi_2= \sigma (fdx_{k}) - fdx_{k}$. Then
$d\psi_{2}= \sigma (dfdx_{k}) - dfdx_{k}.$ 	By virtue of the Lemma~\ref{lem mm}, it is sufficient to check whether $\psi_{2}$ contains monomials of the form $ x_{q}x_{s}^{(2^{t})}dx_i$ and $ x_{q}{(2)}dx_{i}$ for $i, q\in I$. Since $\sigma (fdx_{k}) - fdx_{k} =(\sigma f)(\sigma dx_{k} - dx_{k}) + (\sigma f - f)dx_{k}= (\sigma f)dh_{k} + (\sigma f - f)dx_{k}\in \mathfrak{m}^{2}S\Omega^{1}$, then $\psi_{2}$ does not contain summands of the form $x_{q}^{(2)}dx_{i}$. If $(\sigma f)dh_{k}$ contains $x_{q}x_{s}^{(2^{t})} dx_{i}$, then there are two cases: 1) $dh_{k}$ contains $x_{s}^{(2^{t})}dx_{i}$ and $\sigma f$ contains $x_{q}$, and therefore, $f$ contains $x_{q}$;
2) $dh_{k}$ contains $x_{q}dx_{i}$ and $\sigma f$ contains $x_{s}^{(2^{t})}$. Then $h_{k}$ contains $x_{q}x_{i}$, $q\neq I$ and $\sigma f$ contains $x_{s}^{(2^{t})}$, hence $f$ contains $x_{s}^{(2^{t})}$ and $t>0$. In this case 
$(\sigma f)dh_{k} = x_{s}^{(2^{t})} d(x_{q}x_{i})+ \ldots = x_{q}x_{s}^{(2^{t})} dx_{i}+ x_{i}x_{s}^{(2^{t})} dx_{q} + \ldots = \\ 
dg + x_{q}x_{i}x_{s}^{(2^{t}-1)} dx_{s}+ \ldots $
and $x_{q}x_{i}x_{s}^{(2^{t}-1)} dx_{s}\in \widetilde{\mathfrak{m}}^{(2)} S\Omega^{1}$ 
since $t>0$. If $(\sigma f - f)dx_{i}$ contains $x_{q}x_{s}^{(2^{t})}dx_{i}$, then $\sigma f-f$ contains the product of two indecomposable elements $x_{q}x_{s}^{(2^{t})}$, which is only possible when $f$ contains $x_{k}^{(r)}$. ~Now 
$$\sigma x_{k}^{(r)} - x_{k}^{(r)}~= ~(x_{k} + h_{k})^{(r)}- x_{k}^{(r)} ~= \sum\limits_{s=1}^{r} x_{k}^{(r-s)} h_ {k}^{(s)}.$$ 
Therefore, if $\sigma x_{k}^{(r)} - x_{k}^{(r)}$ contains $x_{q}x_{s}^{(2^{t})}$, then $f$ contains $x_{q}$, but $f \in \mathfrak{m}^{(2)}$. Thus, we can assume that $\psi_{2} \in 
\widetilde{\mathfrak{m}}^{(j+1)}s\Omega^{1}$.
		
	It remains to check $\omega (0)$. ~We have 
$$\sigma (a_{l,l-1}dx_{l-1}dx_{l}+ dx_{l}^{(2)})~- ~(a_{l, l-1}dx_{l-1}dx_{l}+ dx_{l}^{(2)}) ~= $$
$$a_{l,l-1}d(h_{l-1}dx_{l} + h_{l}dx_{l-1}) + a_{l,l-1}dh_{l-1}dh_{l}+ d(h_{l}dx_{l}) + (dh_{l})^{(2)} = $$
$$ d(D\lrcorner(a_{l, l-1}dx_{l-1}dx_{l}+ dx_{l}^{(2)})) +  a_{l, l-1}d (h_{l-1}dh_{l}) + (dh_{l})^{(2)}.$$ 
	~From ~$D\lrcorner \omega(0)\in \widetilde{\mathfrak{m}}^{(j+1)} S\Omega^{1}$ it follows that 
	$$D\lrcorner(a_{l,l-1}dx_{l-1}dx_{l}+ dx_{l}^{(2)})\in \widetilde{\mathfrak{m}}^{(j+1)}S\Omega^{1}.$$
 Since each term in $h_{l-1}dh_{l}$ can be decomposed into a product of three elements and $dx$, $h_{l-1}dh_{l}\in \widetilde{\mathfrak{m}}^{(j+1)}S\Omega^{1}$. Thus, it is sufficient to investigate $(dh_{l})^{(2)}$. But due to Lemma~\ref{lemma h} $(dh_{l})^{(2)}= d\psi_{3}$, where $\psi_{3} \in \widetilde{\mathfrak{m}}^{(j+1)}S\Omega^{1}$, since $h_{l}$ does not contain monomials of the form $\lambda(dz_{r}dz_{s}),$  ~ $x_{q}x_{k}^{(2^{t})}$.
	For $s,k\notin I$ we obtain 
$$\sigma dx_{s}dx_{k} - dx_{s}dx_{k} = d (D\lrcorner dx_{s}dx_{k}) + \ + d (h_{s}dh_{k})$$ 
and $D\lrcorner dx_{s}dx_{k} + h_{s}dh_{k}\in \widetilde{\mathfrak{m}}^{(j+1)}S\Omega^{1}$.
	
	Lemma~\ref{lem 2.3} and the above guarantees that $\psi - D \lrcorner \omega = \psi - \varphi \in \widetilde{\mathfrak{m}}^{(j+2)}S\Omega^{1}$. Also, $\omega' - \sigma\omega = \omega' - \omega - (\sigma\omega - \omega) = d(\varphi - \psi)$.
\end{proof}

\begin{lemma}\label{col 2.5.1}
{\it If $\omega$ is a non-alternating Hamiltonian form and $\omega = \omega(0) + d\varphi + \eta$, where $\varphi \in \widetilde{\mathfrak{m}}^{(2)}S\Omega^{1}$ and $\eta \in H^{2}(\Omega)$, for any $\sigma \in G'$ occurs $\sigma(d\varphi + \eta) - (d\varphi + \eta) = d\psi$, where $\psi \in \widetilde{\mathfrak{m}}^{(2)}S\Omega^{1}$.}
\end{lemma}
\begin{proof}
	The proof of~\ref{lem 2.5.1} shows that for any $\sigma \in G'_{j}$, $j\geqslant 1$, $\sigma (dz_{k}dz_{r})- dz_{k}dz_{r}= d\psi_{1}$ is executed, where $\psi_{1}\in \widetilde{\mathfrak{m}}^{(j+1)}S\Omega^{1}$ , and $\sigma(dfdx_{k}) - dfdx_{k} = d\psi_{2}$, where $\psi_{2}\in \widetilde{\mathfrak{m}}^{(j+1)}S\Omega^{1}$. Take $j= 1$.
\end{proof}

Select two properties of $\omega = \omega(0)+ d\varphi+ \eta$, where $\varphi \in \mathfrak{m}^{(2)} S\Omega^{1}$ and $\eta \in H^{2} (\Omega)$
\begin{equation}\label{equ: eq1}
	\text{ there exists } ~i \in I \text{ such that } m_{i}>1, 
\end{equation}
\begin{equation}\label{equ: eq2}
	m_{i}=1 ~\text{ for any } \, i\in I \text{ and } \eta=\sum_{s,j \notin I} b_{sj}dz_{s}dz_{j}.
\end{equation}

Let us consider the automorphism $\sigma \in G'$ such that  
$$\sigma x_{i} = x_{i} + bx_{s}^{(2^{m_{s}-1})}x_{k}^{(2^{m_{k}-1})}.$$ 
\noindent Since the form $\omega(0)$ is one of canonical form, only two variants of $dx_{i}^{(2)}$ are possible
$$\omega_{1} (0)=\ldots + dx_{i}^{(2)} + \ldots,$$ 
$$\omega_{2} (0)= ~\ldots~+ dx_{i-1}dx_{i} + dx_{i}^{(2)} + \ldots .$$ 
By Lemma~\ref{lem dx2} we have 
$$\sigma \omega_{1} (0)= \omega_{1}(0) + d(bx_{s}^{(2^{m_{s}-1})}x_{k}^{(2^{m_{k}-1})} dx_{i}) + b^{2}dz_{s}dz_{k},$$
 $$\sigma \omega_{2} (0)= \omega_{2}(0) + d(bx_{s}^{(2^{m_{s}-1})} x_{k}^{(2^{m_{k}-1})}) (dx_{i-1} +dx_{i}) + b^{2} dz_{s}dz_{k}.$$ 
 If in the second case we apply automorphism $\sigma_{1}$, which translates $x_{i-1}$ to $x_{i-1} + bx_{s}^{(2^{m_{s}-1})} x_{k}^{(2^{m_{k}-1})}$, we get 
 $$\sigma_{1}\sigma \omega_{2} (0) = \omega_{2}(0) + \\ + d(bx_{s}^{(2^{m_{s}-1})}x_{k}^{(2^{m_{k}-1})}) dx_{i-1}+ b^{2} dz_{s}dz_{k}$$. 
 Thus, only the summand of the form remains 
$d(x_{s}^{(2^{l})}x_{k}^{(2^{t})}dx_{i-1}),$ $i-1 \in I$. Let $d(bx_{s}^{(2^{m_{s}-1})}x_{k}^{(2^{m_{k}-1})}dx_{i}) = d\varphi$, where $i \in I$. Note that $\varphi \notin \widetilde{\mathfrak{m}}^{(2)} S\Omega^{1}$ and prove the following theorem.

\begin{theorem}\label{th sigma}
{\it Let $\omega$ be a non-alternating Hamiltonian form, $\omega = \omega (0) + d\psi + \eta$, where $\eta \in H^{2} (\Omega)$, $\psi \in \mathfrak{m}^{(j+1)}s\Omega^{1}$, $j\geqslant 1$. 
	
	(1) If \eqref{equ: eq1} holds, then there exists an automorphism $\sigma \in G'$ such that $\sigma\omega= \omega(0) + d\widetilde{\psi} + \eta$, where $\widetilde{\psi} \in \widetilde{\mathfrak{m}}^{(j+1)}S\Omega^{1}$.
	
	(2) If $m_{r}= 1$ for all $r \in I$, then there is an automorphism $\sigma \in G'$ such that 
$$\sigma\omega= \omega(0) + \\ + d\widetilde{\psi} + \eta + \sum\limits_{q,i\in I}\sum\limits_{s  \notin I} b_{qsti}^{2^{m_{s} - t+1}} dz_{i}dz_{s}+ \sum\limits_{q,i \in I} b_{iq} {2}dz_{i}dz_{q},$$
 where $\widetilde{\psi} \in \widetilde{\mathfrak{m}}^{(j+1)} S\Omega^{1}$, $b_{qsti}x_{q}x_{s}^ {(2^{t})} dx_{i}$ or $b_{iq}x_{q}x_{i}dx_{s}$ is included in $\psi$.}
\end{theorem}
\begin{proof}
	Let $\varphi = x_{l}^{(2^{r})}x_{s}^{(2^{t})}dx_{i}$, $I\in I$, $0\leqslant r < m_{l}$, $0\leqslant t < m_{s}$ and $\varphi_{1} = x_{q}^{(2)}dx_{i}$, $q\in I$. Except for $l\in I$, $r=0$ ($s\in I$, $t= 0$) and $m_{i}>1$, $i= l$, $r=0$, $t\leqslant 1$, $s\in I$ Lemma~\ref{lem mm} guarantees that $d\varphi= d\widetilde{\varphi}$, where $\widetilde{\varphi} \in \widetilde{\mathfrak{m}}^{(j+1)}S\Omega^{1}$. Therefore, consider $\varphi = x_{q}x_{s}^{(2^{t})}dx_{i}$, where $i,q\in I$. By virtue of the Lemma~\ref{col 2.5.1}, it is sufficient to follow only the action of $\sigma \in G'$ on $\omega(0)$.
	
	(1) Let $m_{q}>1$. Take $\sigma \in G'_{j}$ such that $\sigma x_{i} = x_{i} + x_{q}x_{s}^{(2^{t})}$ and if $\bar{a}_{ii}=a_{i+1, i+1}$, then $\sigma x_{i+1} = x_{i + 1} + x_{q}x_{s}^{(2^{t})}$. Note that for $D \in \mathfrak{g}_{j}'$ associated with $\sigma$ by the formula \eqref{for 1.36}, $D\lrcorner \omega(0) = x_{q}x_{s}^{(2^{t})}dx_{i}$. By corollary~\ref{col dx2} $\sigma \omega(0) = \omega(0) + d\varphi + x_{q}x_{s}^{(2^{t+1}-1)}dx_{q}dx_{s}$, but $x_{q}x_{s}^{(2^{t+1}-1)}dx_{q}dx_{s} = d(x_{q}^{(2)}x_{s}^{(2^{t+1}-1)}dx_{s}) = d\widetilde{\varphi}$ and $\widetilde{\varphi} \in \widetilde{\mathfrak{m}}^{(2j+1)}S\Omega^{1}$, if $\widetilde{\varphi}\neq x_{s}x_{q}^{(2)}dx_{s}$, where $s \in I$. That is $\varphi\neq x_{s}x_{q}dx_{i}$, where $s \in I$. 
	
	Let $m_{k}>1$ for some $k\in I$. Take $\sigma_{1},\sigma_{2} \in G'_{j}$ such that for the associated according to formula \eqref{for 1.36} $D_{1},D_{2} \in \mathfrak{g}_{j}'$, $D_{1}\lrcorner \omega(0) = x_{q}x_{s}^{(2^{t})} dx_{i}$ and $D_{2}\lrcorner \omega(0) = x_{q}x_{s}^{(2^{t})}dx_{k}$. Then by the corollary~\ref{col dx2} we obtain 
$$\sigma_{1}\sigma_{2} \omega(0) = \omega(0) + d(x_{q}x_{s}^{(2^{t})} dx_{i}) + d(x_{q}x_{s}^{(2^{t})} dx_{k}) + d\widetilde {\psi} = $$
$$\omega(0) + d\varphi + \\ + d(x_{q}x_{s}^{(2^{t})}dx_{k})+ d\widetilde{\psi}$$
where $\widetilde{\psi} \in \widetilde{\mathfrak{m}}^{(2)} S\Omega^{1}$ by Lemma~\ref{col 2.5.1}. Now $$d(x_{q}x_{s}^{(2^{t})}dx_{k}) ~ = ~x_{s}^{(2^{t})} dx_{k} dx_{q} ~+ x_{q}x_{s}^{(2^{t}-1)} dx_{k}dx_{s} =$$ $$= d ( x_{k}x_{s}^{(2^{t})} dx_{q} + x_{k}x_{q}x_{s}^{(2^{t}-1)} dx_{s})$$
where $x_{k}x_{q}x_{s}^{(2^{t}-1)} dx_{s} \in \widetilde{\mathfrak{m}}^{(j+1)} s\Omega^{1}$, except for $t= 0$, $s \in I$. That is $\varphi\neq x_{s}x_{q}dx_{i}$, where $s \in I$. The previous case holds for $x_{k}x_{s}^{(2^{t})}dx_{q}$. If $k=s$, then $d (x_{q}x_{s}^{(2^{t})} dx_{s}) = x_{s}^{(2^{t})}dx_{s}dx_{q}= d (x_{s}^{(2^{t}+1)} dx_{q})$ and $x_{s} {(2^{t}+1)}dx_{q} \in \widetilde{\mathfrak{m}}^{(j+1)}S\Omega^{1}$, except for $t= 0$.
		
	Now consider $\varphi= x_{s}x_{q}dx_{i}$, where $i,q, s\in I$. If $m_{q}>1$, then, as we have already found out, $\varphi$ can be replaced with $x_{s}x_{q}^{(2)}dx_{s}$. If $m_{s}>1$, then similarly $x_{s}x_{q}^{(2)} dx_{s}$ is replaced by $x_{s}^{(2)} x_{q}^{(3)} dx_{q}\in \widetilde{\mathfrak{m}}^{(j+1)} S\Omega^{1}$. If $m_{s}=1$, then replace $x_{s}x_{q}^{(2)} dx_{s}$ with $x_{s}x_{q}^{(2)}dx_{q}$ and $d (x_{s}x_{q}^{(2)} dx_{q}) =  x_{q}^{(2)} dx_ {s}dx_{q}= d(x_{q}^{(3)}dx_{q})$, where $x_{q}^{(3)}dx_{q} \in \widetilde{\mathfrak{m}}^{(j+1)}S\Omega^{1}$. If $m_{q}=m_{s}=1$ but there is $k\in I$ with $m_{k}>1$ then replace $x_{s}x_{q}dx_{i}$ with $x_{s}x_{q}dx_{k}$ and $d(x_{s}x_{q}dx_{k})= x_{s}dx_{q}dx_ {q} dx_{k} + x_{q}dx_{s}dx_{k} = d (x_{k}x_{s}dx_{q} + x_{k}x_{q}dx_{s})$. The case $\varphi= x_{k}x_{s}dx_{q}$, where $m_{k}>1$ has already been considered.
	
	The remaining case is $\varphi_{1} = x_{q}^{(2)} dx_{i}$, $q, i\in I$. Then $m_{q}>1$ and $d\varphi = x_{q}dx_{i}dx_{q} = d(x_{q}x_{i}dx_{q})$, which reduces to the previous case. 
	
(2) Let $m_{r}=1$ for all $r\in I$. Since $d\varphi \neq d\widetilde{\varphi}$   for $\widetilde{\varphi} \in \widetilde{\mathfrak{m}}^{(2)} S\Omega^{1}$, by virtue of the~\ref{lem 2.5.1} the degree of $d\varphi$ cannot be raised or lowered by transformations that do not change the cohomology class $\omega$. The only $\sigma \in G'$ that does not lower the degree of $bd\varphi$ results in $b^{2}x_{q}x_{s}^{(2^{t+1}-1)}dx_{q}dx_{s}$ (see corollary~\ref{col dx2}). If $s \notin I$, then using the chain of transformations, change $d(bx_{q}x_{s}^{(2^{t})}dx_{i})$ by $b^{2^{m_{s}-t+1}}dz_{q}dz_{s}$.
	If $s\in I$, then by Lemma~\ref{lem dx2} there is $\sigma \in G'$ such that $\sigma \omega(0) = \omega(0) + d(bx_{s}x_{q}dx_{i}) + b^{2}dz_{s}dz_{q}$.
	
	As a result, we have a transformation $\sigma \in G'$ that translates $\omega$ to $\omega(0) + d\widetilde{\psi} + \eta$ in the case of (1) and translates $\omega$ to $\omega(0) + d\widetilde{\psi} + \eta + \sum cdz_{q}dz_{s}$ in the case of (2). Now, if necessary, using Lemma~\ref{lem 2.5.1} raise the degree $\widetilde{\psi}$ to $j+1$.
\end{proof}  

\begin{lemma}\label{lem 2.5}
{\it Let $\omega, \omega'$ be two non-alternating Hamiltonian forms, and 
$\omega' - \omega = d\varphi+ \eta$, where $\varphi \in \widetilde{\mathfrak{m}}^{(j+1)}s\Omega^{1}$, $j\geqslant 1$, $\eta \in H^{2}(\Omega)$ and $\lambda(\eta)\in \mathfrak{m}^{(j+1)}$. Then if for $\omega$ and $\omega'$  \eqref{equ: eq1} or \eqref{equ: eq2} is holds, there exists $\sigma \in G_{j}'$ such that $\omega' - \sigma\omega = d\widetilde{\varphi} + \widetilde{\eta}$ for some $\widetilde{\varphi} \in \widetilde{\mathfrak{m}}^{(j+2)}S\Omega^{1}$ and $\lambda(\widetilde{\eta})\in \mathfrak{m}^{(j+2)}$.}
\end{lemma}
\begin{proof}
	We fix $i$ from the property \eqref{equ: eq1} in the first case and $i \in I$ if \eqref{equ: eq2} is holds. Since $i_{\omega}$ is an isomorphism of $O$ - modules, there is $D \in \mathfrak{g}_{j}'$ such that $D\lrcorner \omega= \lambda(\eta)dx_{i}$. Proposition~\ref{prop 1.24} clause guarantees the existence of $\sigma_{1} \in G_{j}'$ such that \eqref{for 1.36} is satisfied. By Lemma~\ref{lem 2.3} ~$\sigma_{1}\omega - \omega =  d\psi + \nu$ for some $\psi \in S\Omega^{1}$ and $\nu \in H^{2} (\Omega)$ such that $\lambda(\nu)\in \mathfrak{m}^{(j+1)}$, with $\lambda(\eta) - \lambda (\nu)\in \ \ in \mathfrak{m}^{(j+2)}$. 
	By theorem~\ref{th sigma}, we can assume that $\sigma_{1}\omega - \omega = d\widetilde{\psi} + \nu$, where $\widetilde{\psi} \in \widetilde{\mathfrak{m}}^{(j+1)}S\Omega^{1}$. Then from Proposition~\ref{lem 2.5.1} there is $\sigma_{2} \in G_{j}'$ such that $\sigma_{2}\sigma_{1}\omega - \sigma_{1}\omega = d\widetilde{\psi}_{1}$ where $\widetilde{\psi}_{1} \in \widetilde{\mathfrak{m}}^{(j+2)}S\Omega^{1}$ and $\widetilde{\psi}_{1}- D \lrcorner \omega = \widetilde{\psi}_{1}- \widetilde{\psi} - \varphi\in \widetilde{\mathfrak{m}}^{(j+2)}S\Omega^{1}$, by Lemma~\ref{lem 2.3}.
	We obtain that $\omega' - \sigma_{2}\sigma_{1}\omega = (\omega' - \omega) - (\sigma_{1}\omega - \omega)- (\sigma_{2} \sigma_{1}\omega - \sigma_{1} \omega) = d (\varphi - \widetilde {\psi} - \widetilde{\psi}_{1})+ \eta - \nu = d\widetilde{\varphi} + \widetilde{\eta}$, where $\lambda(\widetilde{\eta}) \in \mathfrak{m}^{(j+2)}$ and $\widetilde{\varphi} \in \widetilde{\mathfrak{m}}^{(j+2)}S\Omega^{1}$.
\end{proof}

\begin{collor}\label{col lem 2.5}
{\it let $\omega, \omega'$ be two non-alternating Hamiltonian forms, with $\omega' - \omega = d\varphi+ \eta$, where $\varphi \in \mathfrak{m}^{(j+1)}s\Omega^{1}$, $j\geqslant 1$, $\eta \in H^{2}(\Omega)$ and $\lambda(\eta)\in \mathfrak{m}^{(j+1)}$. Then if for $\omega$  \eqref{equ: eq1} holds, there exists $\sigma \in G_{j}'$ such that $\omega' - \sigma\omega = d\widetilde{\varphi} + \widetilde{\eta}$ for some $\widetilde{\varphi} \in \widetilde{\mathfrak{m}}^{(j+2)}S\Omega^{1}$ and $\lambda(\widetilde{\eta})\in \mathfrak{m}^{(j+2)}$.}
\end{collor}

Everywhere in what follows assume that for $\omega \in S\Omega^{2}$ the transformation $\sigma \in G'$ from the theorem~\ref{th sigma} is performed.

\begin{proposition}\label{lem 2.6}
{\it (1) If \eqref{equ: eq1} is satisfied for a non-alternating Hamiltonian form $\omega$, then its orbit with respect to $G_{j}'$, $j\geqslant 1$ consists of the forms 
	$$\left\lbrace \omega + d\varphi + \eta ~ | ~ \varphi \in \mathfrak{m}^{(j+1)}S\Omega^{1}, ~\eta \in H^{2} (\Omega), ~\lambda (\eta)\in \mathfrak{m}^{(j+1)} \right\rbrace.$$
	
	(2) If \eqref{equ: eq2} is satisfied for a non-alternating Hamiltonian form $\omega$, then in its orbit with respect to $G_{j}'$, $j\geqslant 1$ lie the forms 
	$$\left\lbrace \omega + d\varphi + \eta ~ | ~ \varphi \in \widetilde {\mathfrak{m}}^{(j+1)} S\Omega^{1}, ~\eta \in H^{2}(\Omega), ~\lambda(\eta)\in \mathfrak{m}^{(j+1)}, \text{ and \eqref{equ: eq2} holds }  \right\rbrace.$$
	
	(3) in the orbit of a non-alternating Hamiltonian form $\omega$ with respect to $G_{j}'$, $j\geqslant 1$ lie the forms 
	$\left\lbrace \omega + d\varphi ~ | ~ \varphi \in \widetilde{\mathfrak{m}}^{(j+1)}s\Omega^{1} \right\rbrace$.}
\end{proposition}
\begin{proof}
	The orbit $G_{j}' \omega$ is contained in the set of forms specified in (1) by virtue of the Corollary~\ref{col 2.4}. Prove that for any non-alternating Hamiltonian form $\omega \in S\Omega^{2}$ and an integer $k\geqslant 1$
	\begin{equation}\label{for 2.4}
	\left\lbrace \omega + d\varphi + \eta ~ | ~ \varphi \in \mathfrak{m}^{(k+1)} S\Omega^{1}, ~\eta \in H^{2} (\Omega), ~\lambda(\eta)\in \mathfrak{m}^{(k+1)} \right\rbrace\subseteq G_{k}' \omega 
	\end{equation}
	by induction on $k$ from above. For $k$ large enough $\mathfrak{m}^{(k+1)} =0$ and the set on the left of the inclusion \eqref{for 2.4} consists of one form $\omega$, so everything is obvious. Let \eqref{for 2.4} be proved for $k= j+1$, where $j\geqslant 1$. Let $\varphi \in \mathfrak{m}^{(j+1)}S\Omega^{1}$ and $\lambda(\eta)\in \mathfrak{m}^{(j+1)}$. By corollary~\ref{col lem 2.5} applied to $\omega$ and $\omega'= \omega + d\varphi + \eta$, there exists $\sigma \in G_{j}'$, $\widetilde{\varphi} \in \widetilde{\mathfrak{m}}^{(j+2)}S\Omega^{1}$ and $\widetilde{\eta} \in H^{2}(\Omega)$, $\lambda(\widetilde{\eta})\in \mathfrak{m}^{(j+2)}$ such that $\omega' - \sigma\omega = d\widetilde{\varphi} + \widetilde{\eta}$. By induction assumption $\omega' \in G_{j+1}' \sigma\omega \subseteq G_{j}' \omega$.
	
	The case (2) is proved similarly if we replace the corollary~\ref{col lem 2.5} with the Lemma~\ref{lem 2.5} and take $\varphi \in \widetilde{\mathfrak{m}}^{(j+1)} S\Omega^{1}$. The case (3) is also proved similarly if we replace the corollary~\ref{col lem 2.5} with the proposition~\ref{lem 2.5.1}, take $\varphi \in \widetilde{\mathfrak{m}}^{(j+1)} S\Omega^{1}$ and  $\eta =0$.
\end{proof}

\begin{proposition}
{\it Let for non-alternating Hamiltonian forms $\omega, \omega'$  \eqref{equ: eq1} be fulfilled. Then $\omega$ and $\omega'$ are conjugated with respect to $G'$ if and only if their images coincide in $s\Omega^{2} {\big /} \mathfrak{m}s\Omega^{2}$.}
\end{proposition}
\begin{proof}
	Since the mapping $S\Omega^{2} \rightarrow S\Omega^{2} {\big /} \mathfrak{m}S\Omega^{2}$ equivariant with respect to $G'$, and $G'$ acts trivially in the space $S\Omega^{2} {\big /} \mathfrak{m}S\Omega^{2}$, then the images of conjugated forms in the space are the same.
	
	Now, let $\omega, \omega'$ have the same initial terms. In particular, $\omega' - \omega = d\psi + \eta$, where $\psi \in \mathfrak{m}S\Omega^{1}$ and $\eta \in H^{2}(\Omega)$. Write $\psi= \sum\limits_{i, j=1}^{n} b_{ij}x_{i}dx_{j} + \varphi$, where $\varphi \in \mathfrak{m}^{(2)}S\Omega^{1}$, $b_{ij}\in K$. Then $d\psi \equiv \sum\limits_{i, j=1}^{n} b_{ij}dx_{i}dx_{j} ~(\text{mod}\, \mathfrak{m}S\Omega^{2})$. Since the images $\omega, \omega'$ in $S\Omega^{2} {\big /} \mathfrak{m}S\Omega^{2}$ coincide, $b_{ij} = b_{ji}$. Since $\varphi = \psi + d\left (\sum\limits_{i<j} b_{ij}x_{i}x_{j}\right) + \sum b_{ii}x_{i}dx_{i}$, $\omega' - \omega = d\varphi + \eta$, but already $\varphi \in \mathfrak{m}^{(2)}s\Omega{1}$. It remains to use the case (1) of Proposition ~\ref{lem 2.6} with the value $j = 1$.
\end{proof}	
\begin{theorem}\label{ass}
{\it Let $\omega$ be a non-alternating Hamiltonian form, 
	$$\omega = \omega(0) + d\varphi + \sum\limits_{i<j} b_{ij}dz_{i}dz_{j},$$
 where $\varphi \in \widetilde{\mathfrak{m}}^{(2)}S\Omega^{1}$ and $b_{ij}\in K$. Then 
	
	(1) if \eqref{equ: eq1} or \eqref{equ: eq2} is satisfied, then $\omega$ is conjugated with respect to $G'$ to the form $\omega(0)$.
	
	(2) $\omega$ is conjugated with respect to $G'$ to the form $\omega(0) + 
	\sum\limits_{i \in I}\sum\limits_{j \notin I} b_{ij}dz_{i}dz_{j} + \sum\limits_{i<j \in I} b_{ij}dz_{i}dz_{j}$.}
\end{theorem}

The theorem follows from Proposition ~\ref{lem 2.6}.

Let $\omega = \omega(0) + \sum\limits_{i<j} b_{ij}dz_{i}dz_{j}$, where $b_{ij}\in K$ and $\omega(0)$ have a canonical form, \eqref{equ: eq1} and \eqref{equ: eq2} do not hold, that is $m_{i}=1$ for all $i \in I$ and  $ b_{sr} = 0$ for $s, r \notin I$.

\begin{remark}
	\begin{itemize}
		\item[$1.$] If $\omega$ has a summand $dx_{i}dx_{j}$, $i, j \notin I$, with $m_{i} < m_{j}$ and summands $b_{si}dz_{s}dz_{i} + b_{sj}dz_{s}dz_{j}$, where $b_{si}, b_{sj} \neq 0$, $s \in I$, one can get rid of the last term by replacing the variables $x_{i} = x_{i} + (\tilde{b}_{sj}/\tilde{b}_{si})x_{j}$, where $\tilde{b}_{si}^{2^{m_{i}}} = b_{si}$ and $\tilde{b}_{sj}^{2^{m_{i}}} = b_{sj}$. This replacement does not change the canonical form of $\omega(0)$. 
		\item[$2.$] Form  $\omega = \omega(0) + \sum\limits_{i \in I} b_{ij}dz_{i}dz_{j}$, where $j \notin I$, is conjugate to the form $\omega = \omega(0) + c_{sj}dz_{s}dz_{j}$ with respect to $G'$, where $s=\min\{ i ~|~ i \in I, ~b_{ij} \neq 0 \}$ and $c_{sj}= \sum\limits_{i} b_{ij}$. Make a transformation $\sigma$, such that for corresponding $D \in \mathfrak{g}'$ ( see \eqref{for 1.36}), $D\lrcorner \omega = \lambda\left( \sum\limits_{i\neq s} b_{ij}dz_{i}dz_{j}\right) dx_{s} + \sum\limits_{i} \lambda( b_{ij}dz_{s}dz_{j}) dx_{i}$, $i \in I$.
		\item[$3.$] Form $\omega = \omega(0) + bdz_{i}dz_{s} + bdz_{j}dz_{s}$, where $i,j,s \in I$, is conjugate to the form $\omega = \omega(0) + bdz_{i}dz_{j}$ with respect to $G'$. Make the transformation $\sigma \in G'$ such that for corresponding  $D \in \mathfrak{g}'$ (\eqref{for 1.36}) ~$D\lrcorner \omega = \lambda(bdz_{i}dz_{s})dx_{j} + \lambda(bdz_{j}dz_{s})dx_{I} + \lambda(bdz_{I}dz_{j}) dx_{s}$.
	\end{itemize}
	In all cases, terms of the form $d\varphi$ are not essential by virtue of the theorem~\ref{ass}.
\end{remark}

\section{ The simplicity  of graded Lie algebras}

We write the form in canonical form 
\begin{equation}\label{kanon}
\omega= dx_{1}dx_{2}+\ldots +dx_{2r-1}dx_{2r}+  \varepsilon_{1}dx_{2}^{(2)}+ \ldots +\varepsilon_{2r-1}dx_{2r}^{(2)}  +dx_{2r+1}^{(2)}+\ldots +dx_{n}^{(2)}, 
\end{equation}
where $r$ is an integer satisfying the inequality $0\leqslant r\leqslant n / 2$, and $\varepsilon_{j}=0$ or $\varepsilon_{j}=1$.

 Recall that $P(n,\overline{m},\omega)$ is identified with the space $O(n,\overline{m})/K $ equipped with the Poisson bracket $~\{f,g \}= \sum\limits_{i,j=1}^{n}\overline{a}_{ij}\partial_{i}f\partial_{j}g \}$.  

\begin{theorem}
{\it If there is $x_{i}$ such that $m_{i}>1$ and $\overline{a}_{ii}\neq 0$, then $P(n,\overline{m},\omega)$ is a simple Lie algebra of dimension $2^{m}-1$. If $m_{i}=1$ for all $i$ such that $\overline{a}_{ii}\neq 0$, then for $n>3$ or for $n=2,3$, $\overline{m}\neq \overline{1}$,  $[P(n,\overline{m},\omega), P (n,\overline{m},\omega)]]$ is a simple Lie algebra of dimension $2^{m}-2$.}
\end{theorem}

\begin{proof}
	Denote $P (n,\overline{m},\omega) = L$. Let $L=L_{-1}+L_{0}+\ldots +L_{r}$ be the standard grading of $L$. Let $y_{i}\in L$ be such that $ad\, y_{i}=\partial_{i}$, $i=1,\ldots, n$. Suppose that $I$ is a nonzero ideal in $L$ and $0\neq f\in I$. By commuting $f$ with $y_{i}$ we obtain that $I\cap L_{-1}\neq 0$. Now $L_{0}$-module $L_{-1}$ irreducibility implies that $L_{-1}\subset I$. Hence we obtain, $L_{i}\subset I$ for $i<r$.
	
	Let $\delta=(2^{m_{1}}-1, \ldots, 2^{m_{n}}-1)$. If $\overline{a}_{kk}\neq 0$, then $\{x_{k}^{(2)},x^{(\delta)}\}=x^{(\delta)}\in I$. \smallskip	
	We obtain that if there is $i$ such that $m_{i}>1$ and $\overline{a}_{ii}\neq 0$, then $I=L$, i.e. the algebra is simple.
	
	Suppose that $m_{i}=1$ if $\overline{a}_{ii}\neq 0$. Then $x_{i}^{(2)} \notin L$. Let's check whether the element $x^{(\delta)}$ is in $I$. To do this, we prove that for any monomials 
$f= x^{(\alpha)}, g= x^{(\beta)}$, $\partial_{s}f\partial_{k}g+ \partial_{k}f\partial_{s}g \neq x^{(\delta)}.$	
\begin{equation}
	\partial_{s}f\partial_{k}g= x^{(\alpha - \varepsilon_{s})}\cdot x^{(\beta-\varepsilon_{k})}= 
	\prod\limits_{j\neq s,k}\dbinom{\alpha_{j}+\beta_{j}}{\alpha_{j}}\dbinom{\alpha_{k}+\beta_{k}-1}{\alpha_{k}}\dbinom{\alpha_{s}+\beta_{s}-1}{\alpha_{s}-1} x^{(\delta)},
\end{equation}
that is $\alpha_{j}+\beta_{j}=2^{m_{j}}-1$, $\alpha_{k}+\beta_{k}-1=2^{m_{k}}-1$ and $\alpha_{s}+\beta_{s}-1=2^{m_{s}}-1$. 
	Obviously, $\dbinom{2^{m_{j}} - 1}{\alpha_{j}}=\dbinom{2^{m_{k}} - 1} {\alpha_{k}} = \dbinom{2^{m_{s} - 1}} {\alpha_{s}-1}=1$. Therefore, $\partial_{s}f\partial_{k}g+ \partial_{k}f\partial_{s}g = (1+1) x^{(\delta)} =0$. From here we get that 			
	$$\{f, g\}= \sum\limits_{s, j=1}^{n}\overline{a}_{sj}\partial_{s}f\partial_{j}g = \sum\limits_ {\overline{a}_{ii}\neq 0} \partial_{i}f\partial_{i}g,$$	
and since $m_{i}=1$, either $\partial_{i}f=0$, or $\partial_{i}g=0$, or $x_{i}^{2}=0$ occurs.		
	Thus, $x^{(\delta)}\notin I$. So $L$ is not a simple algebra and $[L, L] = \langle x^{(\alpha)}, ~\alpha\neq \delta \rangle $. 	\newline
	Let now $I$ is a non-zero ideal of $[L, L]$. We have $L_{i}\subset I$ for $i<r-1$. 
If $n>3$, then there are different $x_{i}x_{k}, x_{i+1}x_{k}\in I$, where $i$ is odd. It follows from \eqref{kanon} that
    \begin{equation*}
	\{ x^{(\delta - \varepsilon_{k})}, x_{k}x_{i} \}= \overline{a}_{ii}x^{(\delta - \varepsilon_{i})}+ \overline{a}_{i,i+1}x^{(\delta - \varepsilon_{i+1})},
	\end{equation*}
	\begin{equation*}
	\{ x^{(\delta - \varepsilon_{k})}, x_{k}x_{i+1} \}= \overline{a}_{i+1,i+1}x^{(\delta - \varepsilon_{i+1})}+ \overline{a}_{i,i+1}x^{(\delta - \varepsilon_{i})}
	\end{equation*}
and $\overline{a}_{ii}, \overline{a}_{i+1,i+1}, \overline{a}_{i,i+1}$ are not simultaneously equal zero and one. Thus, $I=[L, L]$.\newline
If $n=3$ and $\overline{m}\neq \overline{1}$, then $x_{1}^{(\delta_{1})}x_{2}^{(\delta_{2})} = \{ x_{1}^{(\delta_{1})}x_{3}, x_{2}^{(\delta_{2})}x_{3} \}$, $x_{1}^{(\delta_{1}-1)}x_{2}^{(\delta_{2})}x_{3} = \{ x_{1}^{(\delta_{1}-1)}x_{2}^{(\delta_{2})}x_{3}, x_{1}x_{2} \}$ and $$x_{1}^{(\delta_{1})}x_{2}^{(\delta_{2}-1)}x_{3} = \{ x_{1}^{(\delta_{1})}x_{2}^{(\delta_{2}-1)}x_{3}, x_{1}x_{2} \} + a_{22}x_{1}^{(\delta_{1}-1)}x_{2}^{(\delta_{2})}x_{3}$$.
That is $I=[L, L]$.\newline
	If $n=2$ and $m_{2}\neq 1$, then $$x_{2}^{(\delta_{2})} = \{ x_{1}x_{2}^{(\delta_{2}-1)}, x_{2}^{(2)} \}$$ and $x_{1}x_{2}^{(\delta_{2}-1)} = \{ x_{1}x_{2}^{(\delta_{2}-1)}, x_{2}x_{2} \} + x_{2}^{(\delta_{2})}$. That is $I=[L, L]$.
Hence, $[L, L]$ is a simple algebra of dimension $2^{m}-2$.
\end{proof}

\section{ The invariance of the standard filtration}

Let $P^{(1)}(n,\overline{m}, \omega)\subseteq \mathscr{L} \subseteq \widetilde{P}(n,\overline{m}, \omega)$, i.e. $\mathscr{L}$ be a non-alternating Hamiltonian Lie algebra, $\{\mathscr{L}_{i}\}$ standard filtering and $L =\text{gr}\,\mathscr{L} = \bigoplus L_{i}$ its associated graded Lie algebra, which is a non-alternating Hamiltonian Lie algebra corresponding to the form $\omega(0)$. $\widetilde{P} (n,\overline{m}, \omega)$ is identified with the algebra $\widetilde{O} (\mathscr{F})\big/ K$, and $P (n,\overline{m}, \omega)$ is identified with the algebra $O (\mathscr{F})\big/ K$ equipped with Poisson bracket. In what follows we assume that
\begin{equation}\label{size}
	\begin{array}{c}
	n>4 \text{ or } \\
	n=2,3 \text{ and } m_{i}>1, m_{j}>1 \text{ for } I\neq j, \text{ or } \\
	n=4 \text{ and } m_{i}>1, m_{j}>1,m_{k}>1 \text{ for various } i,j, k.
	\end{array}
\end{equation}

Further, in the graded case, we will use the dual form $\overline{\omega} (x, y)$ on $E$, which coincides with the Poisson bracket $\{ x, y \}$.

Similarly ~\cite{Sk2} we introduce an invariant set 
$$\mathfrak{N}(\mathscr{L})= \{ D\in \mathscr{L} ~|~ (\text{ad}\,D)(\text{ad}\,D') \text{ nilpotent for any } D'\in \mathscr{L} \}.$$
Next we need a number of lemmas proved for the classical Hamiltonian case in ~\cite{Sk2}, Chapter III.
\begin{lemma}
{\it For $\mathfrak{N}(\mathscr{L})$ the following inclusions $\mathscr{L}_{2} \subset \mathfrak{N}(\mathscr{L}) \subset \mathscr{L}_{0}$ are fulfilled.}
\end{lemma}
\begin{proof} 
	 If $D\in \mathscr{L}_{2}$ and $D'\in \mathscr{L}$, then $(\text{ad}\, D) (\text{ad}\, D')\mathscr{L}_{j} \subset \mathscr{L}_{j+1}$ for all $j$, whence the nilpotence of the operator $(\text{ad}\,D)(\text{ad}\, D')$ follows, that is $\mathscr{L}_{2} \subset \mathfrak{N}(\mathscr{L})$.
	
	Suppose now that $D\notin \mathscr{L}_{0}$. For $D'\in \mathscr{L}_{1}$ and $j\geqslant -1$, the operator $(\text{ad}\, D)(\text{ad}\, D')$ induces a linear transformation of the space $L_{j}= \mathscr{L}_{j} \big / \mathscr{L}_{j+1}$, which is nilpotent if only $D \in\mathfrak{N}(\mathscr{L})$. Therefore, when checking the inclusion of $\mathfrak{N} (\mathscr{L}) \subset \mathscr{L}_{0}$, we can proceed to the associated graded algebra.
	
	Thus, it is enough to show that for each $0\neq D\in L_{-1}$ there exist $D' \in L_{1}$ and $D'' \in L_{j}$ for suitable $j$, that $[D, [D', D'']]= \lambda D''$, where $\lambda\neq 0$, i.e. the operator $(\text{ad}\,D)(\text{ad}\,D')$ is not nilpotent on $L_{j}$. Let $D= x$, where $0\neq x\in E$. Consider first the case of $n> 2$. If $\{ x, x \} =1$ and there are $x'\neq y' \in \langle x \rangle^{\bot}$ such that $\{ x', y' \} =1$, $\{ x', x'\} =0$, \ \ then in this case $D' =xx'y'$, $D''=x'$ and $ 
	\{ x, \{ xx' y', x' \} \} =\{ x, xx' \} = x'$. Otherwise, we have different vectors $x', y' \in \langle x \rangle^{\bot}$ of unit length. Suppose $D' =xx'y'$, $D''=x'+y'$, then $\{ x, \{ xx'y', x'+y' \} \} =\{ x, xy'+xx' \} = x'+y'$. If $\{ x, x \} =0$, then take $y\in E$ such that $\{ x, y \} =1$. For $n> 3$ take different $x', y' \in \langle x, y \rangle^{\bot}$ and if $\{ x', y' \} =1$, 
	$\{ x', x' \} =0$, then $D'=yx'$, $D''=x'$ and $\{ x, \{ yx 'y', x' \} \} = \{ x, yx' \} = x '$ if $\{ x', x' \} = \{ y', y' \} =1$ then $D' =yx' y$, $D''=x'+y'$ and $\{ x, \{ yx' y', x'+y' \} \} = \{ x, yy'+yx' \} = x'+y'$. For $n= 3$ we have $z \in \langle x, y \rangle^{\bot}$. If $m_{z}>1$, then $\{ x, \{ yz^{(2)}, z \} \} =\{ x, yz \} = z$. If $m_{z}=1$, then $m_{y}>1$ and $\{ x, \{ xy^{(2)}, x \} \} =\{ x, xy \} = x$. The remaining case is $n= 2$. If $\{ x, x \} =0$, $ \{ x, y \} =1$, then $\{ x, 
	\{xy^{(2)}, x \} \} =\{ x, xy \} = x$. If $\{ x, x \} =1$, then $\{ x, \{ x^{(3)}, x \} \} =\{ x, x^{(2)} \} = x$.	
\end{proof}

\begin{lemma}[~\cite{Sk2}]\label{lemmm1.2}
{\it Suppose that $\mathscr{M}\subset \mathscr{L}$ is a subalgebra such that $\mathscr{M}+ \mathscr{L}_{0}= \mathscr{L}$ and $\mathscr{M}\supset \mathscr{L}_{2}$. Then $\mathscr{M}$ contains a nonzero ideal of the algebra $\mathscr{L}$.	\\ [-1.0 cm]} 

	\begin{flushright}$\square$\end{flushright} 
\end{lemma}

\begin{lemma}\label{lemmm1.3}
\it Let $V$ be a vector space, $U$ its own nonzero subspace, $G$  Lie algebra of linear transformations of $V$ and $N_{G}(U)$ subalgebra of those linear transformations from $G$, with respect to which $U$ is stable. Let's put $l=\text{codim}_{V}U$ and $t=\text{codim}_{G}N_{G}(U)$.
	\begin{itemize}
	\item [$(1)$] Let $G= s(V,b)$ be the Lie algebra of transformations preserving the nondegenerate non-alternating symmetric bilinear form $b$. Then $l+t \geqslant n$, and equality is achieved only at $l=1$.	
	\item [$(2)$] Let $G= s(V,b)^{(1)}$. Then $l+t > n$ except for the following cases
	 \begin{itemize}
		\item[$a.$] if $l=1$ and $U\cap U^{\bot}=0$, then $l+t=n$;
		\item[$b.$] if $l=1$ and $U^{\bot}$ is isotropic, then $l+t=n-1$;
		\item[$c.$] if $l=2$, ~$n=3$ and $U$ is isotropic, then $l+t=n$;
		\item[$d.$] if $l=2$, ~$n=5$ and $\dim U\cap U^{\bot}=2$, then $l+t=n$;
		\item[$e.$] if $l=2$, ~$n=4$ and $U$ is totally isotropic, then $l+t=n-1$;
		\item[$f.$] if $l=3$, ~$n=6$ and $U$ is totally isotropic, then $l+t=n$.
	\end{itemize}
\end{itemize}
\end{lemma}
\begin{proof} 
	
		Let's put $r= \dim U\cap U^{\bot}$. Then $r\leqslant \dim U=n-l$ and $r\leqslant \dim U^{\bot}= l$. For $A\in \text{gl}(V)$, denote by $b_{A}$ the bilinear form of $V$ defined by the rule $b_{A}(u,v)= b(Au,v)$, where $u,v\in V$. Matching $A\mapsto b_{A}$ gives a linear isomorphism $\text{gl}(V)$ to the space of all bilinear forms on $V$. The transformation $A$ preserves the form $b$, that is $A\in s(V,b)$, if and only if $b_{A}$ is symmetric, from where $s(V, b)\cong s(V)$ follows, where $s(V)$ is the space of all symmetric bilinear forms on $V$.
	Since $b$ is non-degenerate, the vector $u\in V$ lies in $U$ if and only if $b(u, v)=0$ for all $v\in U^{\bot}$. Hence, $U$ is stable with respect to $A$ if and only if $b(Au,v)=0$, i.e. when $b_{A}(u,v)=0$, for all $u\in U$ and $v\in U^{\bot}$. We denote by $T$ a subspace in $s (V)$ consisting of those bilinear forms with respect to which $U$ and $U^{\bot}$ are orthogonal to each other. Then $N_{s(V,b)} (U)\cong T$ and we see that in the case of (1) $t=\text{codim}_{s(V)}T= l(n-l) - \dbinom{r}{2}$. Thus,
\begin{equation}\label{key1}
	l+t-n= (l-1)(n - l)-\frac{1}{2}r(r-1).
	\end{equation}
	Using inequalities $r\leqslant l$ and $r\leqslant n-l$, we obtain $r(r-1)\leqslant (l-1)(n-l)$ and $l+t-n\geqslant  \frac{1}{2}(l-1)(n-l)$. The last expression is positive if $l> 1$. If $l=1$, then $r=1$ or $r=0$, and the right part of the formula \eqref{key1} vanishes.
	
	Now let $G= s(V,b)^{(1)}\cong s(V)^{(1)}= sk(V)$, where $sk (V)$ is the space of all skew-symmetric bilinear forms on $V$. We denote by $T'$ a subspace in $sk (V)$ consisting of those bilinear forms with respect to which $U$ and $U^{\bot}$ are orthogonal to each other. Then $t=\text{codim}_{sk (V)}T'= l(n-l) - \dbinom{r}{2} - r$ and
	\begin{equation}\label{key2}
		l+t-n= (l-1)(n - l) - \frac{1}{2}r(r+1). 
	\end{equation}
	Using the inequality $r\leqslant l$ and $r\leqslant n-l$, we get $r(r+1)\leqslant (l-1)(n-l)+2(n-l)$ and $l+t-n\geqslant \frac{1}{2}(l-1)(n-l) - (n-l)= \frac{1}{2}(l-3)(n-l)$. The last expression is positive if $l> 3$.
	If $l= 1$, then at $r=0$ the formula \eqref{key2} vanishes, and at $r=1$ it takes the value $-1$. If $l= 2$, the formula \eqref{key2} is $l+t-n= (n - 2) - \frac{1}{2}r(r+1)$. At $r=0$ we get $l+t-n > 0$. For $r=1$ we get $l+t-n= n-3$ and, since $n \geqslant l+r=3$, the expression is positive except for the case $n = 3$. At $r=2$ we get $l+t-n= n-5$ and, since $n \geqslant is 4$, the expression is positive except for the case $n = 4,5$. If $l= 3$, the formula \eqref{key2} is $l+t-n= 2(n - 3) - \frac{1}{2}r(r+1)$. For $r=0$, $r=1$ or $r=2$ we get $l+t-n > 0$. For $r=3$ we get $l+t-n= 2n-12$ and, since $n \geqslant 6$, the expression is positive except for $n = 6$.
\end{proof}

The following theorem corresponds to Proposition 1.4 of Chapter III, ~\cite{Sk2} for the classical Hamiltonian case.

\begin{theorem}\label{col: col9}
{\it Let \eqref{size} conditions be met. Then $\mathscr{L}_{0}$ is the only subalgebra of the smallest codimension among all subalgebras of the Lie algebra $\mathscr{L}$ containing $\mathfrak{N} (\mathscr{L})$ but not containing nonzero ideals of the algebra $\mathscr{L}$. In particular, $\mathscr{L}_{0}$ is an invariant subalgebra in $\mathscr{L}$.}
\end{theorem}

\begin{proof}

	Let $\mathscr{M}\subset \mathscr{L}$ be a subalgebra that contains $\mathfrak{N} (\mathscr{L})$, and therefore $\mathscr{L}_{2}$, but does not contain nonzero ideals of the algebra $\mathscr{L}$. Denote $M=\text{gr}\,\mathscr{M}= \bigoplus M_{i}$. Assume that $\mathscr{M}\neq \mathscr{L}_{0}$. Since $\text{codim}_{L}M = \text{codim}_{\mathscr{L}}\mathscr{M}$, it is sufficient to show that
	$$\text{codim}_{L}M=\sum\dim (L_{i}{\big /}M_{i})> n=\text{codim}_{\mathscr{L}}\mathscr{L}_{0}= \dim L_{-1}.$$
	By virtue of Lemma~\ref{lemmm1.2} $\mathscr{M}+ \mathscr{L}_{0}\neq \mathscr{L}$, where $M_{-1}\neq L_{-1}$. Put $l_{i}= \dim(L_{i}{\big /}M_{i})$.
	
	If $n=2$, then $m_{i}> 1$ for all $i$ and $L_{0} = s(L_{-1}, \omega(0))$. Then by Lemma~\ref{lemmm1.3} (case 1) $l_{-1}+ l_{0}\geqslant n$, the inequality being strict if only $l_{-1}> 1$. Let $x\notin M_{-1}$, $x\in E$, then $x^{(3)}\notin M_{1}$ or $y^{(2)}x\notin M_{1}$ for some $y\in M_{-1}$, hence $l_{-1}+ l_{0}+ l_{1}> n$.
	
	Suppose $n> 2$. Denote $L_{0}' = s(L_{-1}, \omega(0))^{(1)}$. Then $l_{-1}+ l_{0}\geqslant l_{-1}+ \\ + \dim L_{0}'{\big /}(L_{0}'\cap M_{0})> n$ is fulfilled except for the cases listed in the Lemma~\ref{lemmm1.3}. Denote $U=M_{-1}$. Suppose first that $l_{-1}= 1$ and $x\notin U$. The induced bilinear form on $U'= U\cap \langle x \rangle^{\bot}$ is nondegenerate and $n-2\leqslant \dim U' \leqslant n-1$. If $n\geqslant 6$ or $n= 5$, $\dim U' = n-1$ then there are $x_{1}, y_{1}, x_{2}, y_{2} \in U'$, the subspace $\langle x_{1}, y_{1} \rangle$ is orthogonal to the subspace $\langle x_{2}, y_{2} \rangle$ and $\{x_{i}, y_{i} \}=1$ or $\{x_{i}, x_{i} \}=\{y_{i}, y_{i} \}=1$, $i=1, 2$. Then $x_{1}y_{1}x$ and $x_{2}y_{2}x$ are linearly independent modulo $M_{1}$ and do not belong to $M_{1}$. If $n= 5$ and $\dim U' = 3$ then there are $x_{1}, y_{1}, x_{2} \in U'$, the subspace $\langle x_{1}, y_{1} \rangle$ is orthogonal to the subspace $\langle x_{2} \rangle$, $\{x_{1}, y_{1} \}=1$ or $\{x_{1}, x_{1} \}=\{ y_{1}, y_{1} \}=1$ and $\{x_{2}, x_{2} \}=1$. Then $x_{1}y_{1}x$ and $x_{2}y_{1}x$ are linearly independent modulo $M_{1}$. If $n= 4$, then there are three elements of height greater than $1$. Then for $x_{1}, y_{1}\in U'$ we have $m_{x_{1}}> 1$. Hence $x_{1}^{(2)}x$ and $x_{1}y_{1}x$ are linearly independent modulo $M_{1}$.
	
	So $l_{-1}= 2$ or $l_{-1}= 3$. Let us first consider the case $n=6$, and $U$ is totally isotropic and has a dimension of $3$. Then we have elements $x_{1}, x_{2}, x_{3} \in U$ and $y_{1}, y_{2}, y_{3} \in E$ such that $\{ x_{i}, y_{j} \}=\delta_{ij}$, $i, j=1, 2, 3$. Hence $y_{1}y_{2}y_{3}\notin M_{1}$. Now let $n=5$ and $\dim U\cap U^{\bot}=2$. We have $x_{1}, x_{2}, x_{3} \in U$ and $y_{1}, y_{2} \in E$ such that $\{ x_{i}, y_{j} \}=\delta_{ij}$, $i=1, 2, 3$, $j=1, 2$. Hence $y_{1}y_{2}x_{3}\notin M_{1}$. In the case of $n=4$ we have $U$ is totally isotropic and has dimension $2$. Let $x_{1}, x_{2} \in U$ and $y_{1}, y_{2}\in E$ be such that $\{ x_{i}, y_{j} \}=\delta_{ij}$, $i, j=1,2$. Then $m_{y_{1}}> 1$ and $y_{1}^{(3)}$ and $y_{1}^{(2)} y_{2}$ are linearly independent modulo $M_{1}$ elements. The remaining case is $n=3$ and $U$ isotropic and one-dimensional. Let $x_{1} \in U$ and $y_{1}, y_{2}\in E$ be such that $\{ x_{1}, y_{j} \}=\delta_{ij}$, $j=1,2$. Then $m_{y_{1}}> 1$ and $y_{1}^{(3)}\notin M_{1}$.
	
	Thus, in each case, $l_{-1}+ l_{0}+ l_{1}> n$, i.e. $\text{codim}_{\mathscr{L}}\mathscr{M}> \text{codim}_{\mathscr{L}}\mathscr{L}_{0}$.
\end{proof}
Since the natural filtration of $\mathscr{L}$ is defined uniquely by subalgebra $\mathscr{L}_0$ the following statement is obvious.
\begin{collor} 
{\it If conditions \eqref{size} hold, then the natural filtration of $P (n,\overline{m}, \omega)$ is invariant.} 
\end{collor}

Let $P^{(1)}(\mathscr{F}, \omega)\subseteq \mathscr{L} \subseteq \widetilde{P}(\mathscr{F}, \omega)$ and $P^{(1)}(\mathscr{F}', \omega')\subseteq \mathscr{L}' \subseteq \widetilde{P}(\mathscr{f}', \omega')$. If $\varphi\colon \mathscr{L}\rightarrow \mathscr{L}'$ is an isomorphism, then the theorem~\ref{col: col9} implies that $\varphi(\mathscr{L}_{0}) = \mathscr{L}_{0}'$, and the embedding theorem ~\cite{K} implies that $\mathscr{F}= \ mathscr{F}'$ and $\varphi$ is induced by the automorphism $O(\mathscr{F})\rightarrow O(\mathscr{F}')$, which is also denoted by $\varphi$. According to Theorem~\ref{filt 4} $\varphi (\omega) = \omega'$. Now let $\omega =\omega (0), \omega' =\omega'(0),$ i.e. $\mathscr{L} =gr\mathscr{L}, ~\mathscr{L}'= gr\mathscr{L}'$ be graded Lie algebras. Then
$ gr\varphi$ is the homogeneous isomorphism induced by the linear isomorphism of the spaces $E$ and $E',$  
$gr\varphi (\mathscr{F}) = \mathscr{F}'$. Since $L_0$-module $L_{-1}$ is absolutely irreducible, it follows from Schur's Lemma that $gr\varphi (\omega (0))= \omega'(0)$. 

\begin{theorem}
{\it Non-alternating Hamiltonian Lie algebras $\mathscr{L}$ and $\mathscr{L}'$ are isomorphic if and only if there is an admissible isomorphism $ \varphi\colon O(\mathscr{F})\rightarrow O(\mathscr{F}')$ such that $\varphi(\omega)= \omega'$. In particular, 
$$Aut (\mathscr{L})\cong \{ \varphi \in Aut (O (\mathscr{F})) ~|~ \varphi ~\text{is admissible and} ~\varphi (\omega)= \lambda\omega, \lambda \in K \}.$$  If the graded non-alternating Hamiltonian Lie algebras are isomorphic, then there is an admissible linear isomorphism $\varphi, $ such that
$\varphi (\mathscr{F}) = \mathscr{F}'$ and the corresponding non-alternating bilinear forms have the same invariants.}
	\begin{flushright}$\square$\end{flushright}
\end{theorem}

The following examples show that in exceptional cases nontrivial filtered deformations of non-alternating Hamiltonian Lie algebras are possible. In particular, the constructed algebras are new simple filtered Lie algebras of characteristic 2. The authors plan to consider exceptional cases in future publications. Below we announce some results on non-alternating Hamiltonian Lie algebras corresponding to the forms

\begin{equation*} 
\omega_{1}=dx_{1}dx_{2}+ dx_{3}^{(2)}+ a\overline{x_{1}} \overline{x_{2}}dx_{1}dx_{2}+ b\overline{x_{1}} \overline{x_{3}} dx_{1}dx_{3}+ c\overline{x_{2}}\overline{x_{3}}dx_{2}dx_{3}, 
\end{equation*}
\begin{equation*} 
\omega_{2}=dx_{1}^{(2)}+dx_{2}^{(2)}+dx_{3}^{(2)}+ a\overline{x_{1}} \overline{x_{2}}dx_{1}dx_{2}+ b\overline{x_{1}} \overline{x_{3}} dx_{1}dx_{3}+ c\overline{x_{2}}\overline{x_{3}} dx_{2}dx_{3},
\end{equation*}
\begin{equation*} 
\omega_{3}=dx_{1}dx_{2}+dx_{2}^{(2)}+dx_{3}^{(2)}+ a\overline{x_{1}} \overline{x_{2}}dx_{1}dx_{2}+ b\overline{x_{1}} \overline{x_{3}} dx_{1}dx_{3}+ c\overline{x_{2}}\overline{x_{3}} dx_{2} dx_{3}.
\end{equation*}

\begin{theorem}
{\it If either $b\neq0$ or $c\neq0$ or $m_{3}>1$ then $P(3,\overline{m},\omega_{1})$ is a simple Lie algebra of dimension $2^{m}-1$. If $b=c=0$, $m_{3}=1$ and $(m_{1},m_{2})\neq(1,1)$,  $[P(3,\overline{m},\omega_{1}),P(3,\overline{m},\omega_{1})]$ is a simple Lie algebra of dimension $2^{m}-2$.}
\end{theorem}

\begin{theorem}
{\it $P(3,\overline{m},\omega_{2})$ is a simple Lie algebra of dimension $2^{m}-1$ for $(a,b,c)\neq (t,t,t)$, $t\in K$ or for $\overline{m}\neq\overline{1}$.}
\end{theorem}

\begin{theorem}
{\it If either $(m_{1},m_{3})\neq(1,1)$, or $b\neq0$ or $a\neq c$, then $P(3,\overline{m},\omega_{3})$ is a simple Lie algebra of dimension $2^{m}-1$. If $a=c$, $b=0$, $(m_{1}, m_{3})=(1,1)$ and $m_{2}\neq 1$, then $[P(3,\overline{m},\omega_{3}), P(3,\overline{m},\omega_{3})]$ is a simple Lie algebra of dimension $2^{m}-2$.}
\end{theorem}

\end{document}